\documentclass[11pt,a4paper,reqno]{amsart}
\usepackage[margin=1.1in]{geometry}
\usepackage{amssymb,amsmath,graphicx,amsfonts}
\usepackage{mathrsfs}
\usepackage{multirow}
\usepackage{graphicx,epsfig}
\usepackage{color}
\usepackage{enumerate,enumitem}
\usepackage{empheq}
\usepackage{cases}
\usepackage{cite}
\usepackage{mathrsfs}
\usepackage{amsgen}
\usepackage{amscd}

\allowdisplaybreaks

\usepackage{booktabs}
\usepackage{amsmath}
\usepackage{latexsym}
\usepackage{amssymb,amsmath,graphicx,amsfonts,euscript}
\usepackage{amsfonts,bm,hyperref,subcaption}
\usepackage{enumerate}
\newtheorem{theorem}{Theorem}[section]

\theoremstyle{definition}

\newtheorem{example}[theorem]{Example}

\numberwithin{theorem}{section}
\numberwithin{equation}{section}

\def\le{\leqslant}
\def\ge{\geqslant}
\def\Omega{\varOmega}
\def\Delta{\varDelta}
\def\bex{\begin{exercise}\upshape}
\def\eex{\end{exercise}}

\numberwithin{equation}{section}
\begin{document}

\title[]{Dual formulations of geometric curvature flows and their discretizations}\thanks{This work was partially supported by the National Natural Science Foundation of China (Project No. 12525111) and the Research Grants Council of Hong Kong (Project No. 15301321 and RFS2324-5S03).}

\author[]{Guangwei Gao,\,\, Buyang Li,\,\, and\,\,Rong Tang}
\address{Guangwei Gao, Buyang Li, Rong Tang: Department of Applied Mathematics, The Hong Kong Polytechnic University, Hong Kong. 
{\rm Email address: {\tt guang-wei.gao@polyu.edu.hk}, {\tt buyang.li@polyu.edu.hk} and {\tt claire.tang@polyu.edu.hk}
}}




\maketitle

\begin{abstract}
  We propose new formulations of geometric curvature flows---referred to as \emph{dual formulations}---that are equivalent to the original formulations but provide a novel framework for constructing linearly implicit and energy-stable schemes for curvature-driven surface evolution, including mean curvature flow, surface diffusion, and solid-state dewetting on a substrate with a moving contact line. The dual formulations are derived by introducing, at the continuous level, an additional unknown in the form of a dual multiplier. This augmentation does not alter the continuous dynamics but makes the underlying energy-dissipation structure explicit and, in turn, enables a systematic design of linearly implicit discretizations that inherit energy stability. A key feature of this framework is that it accommodates a broad class of artificial tangential motions which can be used to maintain good mesh quality of the computed surfaces. As an illustration, we combine the framework with the minimal-deformation-rate (MDR) tangential motion, leading to what we call the \emph{dual-MDR} scheme. The resulting method is linearly implicit and energy-stable, while retaining the MDR tangential motion to maintain good mesh quality. Extensive numerical experiments demonstrate the convergence of the proposed schemes, their structure-preserving properties, and advantages on representative benchmark problems.
\end{abstract}





\setlength\abovedisplayskip{4pt}
\setlength\belowdisplayskip{4pt}

\section{Introduction}


Geometric curvature flows, which govern curvature-driven surface evolution, have attracted sustained interest in mathematics and physics, largely due to their ability to model a wide range of interfacial phenomena, including crystal growth \cite{allen1979microscopic,tiller1991science}, thin-film and foam morphology \cite{ishida2017hyperbolic,Thompson1}, grain-boundary migration \cite{Gottstein2009}, and surfactant-laden two-phase flows \cite{worner2012numerical}. The development and analysis of stable, convergent numerical methods for curvature flows has therefore become an active area of research in engineering and computational mathematics.

This work concerns the development of structure-preserving numerical methods for computing curvature-driven surface evolution, with a focus on mean curvature flow and surface diffusion on closed surfaces, as well as solid-state dewetting on a substrate with a moving contact line. An evolving surface \(\Gamma(t)\subset\mathbb{R}^3\), \(t\in[0,T]\), is typically determined by a normal-velocity law of the form
\begin{subequations}\label{eq:normal-v}
\begin{align}
v\cdot n &= -H && \text{(mean curvature flow)}, \label{eq:mean curvature flow-vn}\\
v\cdot n &= \Delta_{\Gamma(t)} H && \text{(surface diffusion)}. \label{eq:surface diffusion-vn}
\end{align}
\end{subequations}
where \(H\) and \(n\) denote the mean curvature and unit normal of \(\Gamma(t)\), and $\Delta_{\Gamma(t)} $ denotes the surface Laplace--Beltrami operator. The evolving surface $\Gamma(t)$ with initial condition $\Gamma(0)=\Gamma^0$ can be represented as the image $\Gamma(t)=\{X(p,t):p\in \Gamma^0\}$ of a flow map \(X(\cdot,t):\Gamma^0\to\mathbb{R}^3\) satisfying
$$
\partial_t X(\cdot,t)=v(X(\cdot,t),t)
\,\,\,\text{on}\,\,\, \Gamma^0,\,\,\,\mbox{with}\,\,\, X(\cdot,0)=\mathrm{id}.
$$
In solid-state dewetting, the surface diffusion law \eqref{eq:surface diffusion-vn} is additionally coupled with boundary conditions on the moving contact line between the surface and the substrate. 

A fundamental feature of these curvature flows is their energy-dissipation structure. For a closed surface \(\Gamma(t)\), differentiating the surface-area energy \(|\Gamma(t)|\), applying integration by parts, and using the geometric identity 
\begin{align}\label{eq:id-H}
Hn = -\Delta_{\Gamma}\mathrm{id}\quad\mbox{with}\,\,\,{\rm id}(x)\equiv x ,
\end{align}
one obtains (see \cite[Theorem~32]{BarrettGarckeNurnberg2020}) 
\begin{equation}\label{eq:der-St}
\frac{\mathrm{d}}{\mathrm{d} t} |\Gamma(t)|
= \int_{\Gamma(t)} \nabla_{\Gamma(t)} \cdot v
= \int_{\Gamma(t)} \nabla_{\Gamma(t)} \mathrm{id} \cdot \nabla_{\Gamma(t)} v
= \int_{\Gamma(t)} H n \cdot v .
\end{equation}
Substituting \eqref{eq:mean curvature flow-vn} and \eqref{eq:surface diffusion-vn}
into~\eqref{eq:der-St} gives the energy-dissipation laws
\begin{equation}\label{eq:area-decay}
\frac{\mathrm{d}}{\mathrm{d} t} |\Gamma(t)| =
\begin{cases}
- \displaystyle\int_{\Gamma(t)} |H|^2  \le 0,
& \text{(mean curvature flow)}, \\[2mm]
- \displaystyle\int_{\Gamma(t)} |\nabla_{\Gamma(t)} H|^2  \le 0,
& \text{(surface diffusion)}.
\end{cases}
\end{equation}
In the solid-state dewetting setting, the relevant energy is the total free energy
\(W(t):=|\Gamma(t)|-\cos\theta\,|S_1(t)|\), where \(S_1(t)\) denotes the evolving
film--substrate interface and \(\theta\) is the prescribed contact angle. Analogous to the
closed-surface case, the solid-state dewetting system satisfies the energy-dissipation inequality
\(\frac{\mathrm{d}}{\mathrm{d} t} W(t)\le 0\).

A natural and widely used framework for computing curvature-driven surface evolution is provided by parametric finite element methods (FEMs), initiated by Dziuk in \cite{dziuk1990algorithm} and subsequently extended and refined for a broad range of geometric flows and related problems; see, e.g., \cite{bonito2010parametric,dziuk2008computational,bansch2005finite,elliott2015evolving,kovacs2018higher,deckelnick2018stability,dziuk2007finite,dziuk2013finite,jiang2021perimeter}. In this approach, one computes triangulated surfaces $\Gamma_h^{m}$, $m=1,2,\dots$, by the surface FEM, as approximations to the surface $\Gamma(t_m)$ at discrete time levels $t_m=m\tau$, $m=1,2,\dots$, 
where \(\tau\) denotes the time-step size. Within the parametric finite element framework, two guiding principles have emerged in the modern design of discretizations: preserving the underlying energy dissipation and maintaining good mesh quality. On the one hand, energy-stable schemes tend to be more robust for long-time simulations and in regimes where singular features (such as pinch-off) may form; see Fig.~\ref{fig:Box118-MDR}. On the other hand, curvature-driven evolution can induce large deformations, and without an appropriate mesh-control mechanism the discrete surface may suffer from severe element distortion or even degeneracy, leading to a loss of geometric accuracy and, in extreme cases, breakdown of the computation.



Dziuk's original parametric FEM in \cite{dziuk1990algorithm} advances the discrete surface \(\Gamma_h^{m}\) by a one-step flow map $X_h^{m+1}=\mathrm{id}+\tau v_h^{m+1}:\Gamma_h^{m}\to\Gamma_h^{m+1}$, 
where the discrete velocity \(v_h^{m+1}\) is the finite element solution of the velocity law without tangential motion:
$$
v=-Hn\approx \Delta_{\Gamma}\bigl(\mathrm{id}+\tau v\bigr) .
$$
Specifically, denoting by $S_h(\Gamma_h^{m})$ the Lagrange finite element space on $\Gamma_h^{m}$, the discrete velocity \(v_h^{m+1}\in S_h(\Gamma_h^{m})^3\) is determined from the weak formulation:
\begin{align}\label{Dziuk-weak-eqn} 
\int_{\Gamma_h^{m}} v_h^{m+1}\cdot \eta_h
\;+\;
\int_{\Gamma_h^{m}} \nabla_{\Gamma_h^{m}} X_h^{m+1} \cdot \nabla_{\Gamma_h^{m}} \eta_h
= 0,
\quad \forall\, \eta_h \in S_h(\Gamma_h^{m})^3.
\end{align} 
The convergence of Dziuk's scheme has been established for curve and surface evolution under various conditions (e.g., bounded deformation and the use of higher-order finite elements); see
\cite{deckelnick1995approximation,dziuk1994convergence,li2020convergence,ye2021convergence,li2021convergence,bai2023new}. In addition, the method inherits the energy stability. This can be seen by choosing $\eta_h=v_h^{m+1}$ in \eqref{Dziuk-weak-eqn} and utilizing the following geometric inequality (see \cite[(2.21)]{barrett2008hypersurfaces} or \cite[(2.31)]{barrett2007parametric}): 
\begin{align}\label{area-ineq}
\int_{\Gamma_h^{m}} \nabla_{\Gamma_h^{m}} X_h^{m+1} \cdot \nabla_{\Gamma_h^{m}} (\tau v_h^{m+1})
\;\ge\;
|\Gamma_h^{m+1}|-|\Gamma_h^{m}|.
\end{align}
However, since no tangential motion is included to control the distribution of mesh points, the evolving mesh can become highly skewed and may even degenerate when the surface experiences large deformations. 

Since the surface shape is determined only by the normal velocity, Barrett, Garcke, and N\"urnberg \cite{barrett2007parametric,barrett2008hypersurfaces,barrett2008willmore} proposed a variational framework that augments the evolution by an artificial tangential velocity for maintaining good mesh quality in the computed surfaces. For mean curvature flow, the BGN method can be equivalently written as finding a one-step flow map $X_h^{m+1}=\operatorname{id}+\tau v_h^{m+1}:\Gamma_h^{m}\rightarrow \Gamma_h^{m+1}$, determined by a surface velocity \(v_h^{m+1}\in S_h(\Gamma_h^{m})^3\), such that
    \begin{align}\label{BGN-weak-eqn}
        \int_{\Gamma_h^{m}}^{(h)} (v_h^{m+1} \cdot n_h^{m})\,(\eta_h \cdot n_h^{m})+ 
        \int_{\Gamma_h^{m}} \nabla_{\Gamma_h^{m}} X_h^{m+1} \cdot \nabla_{\Gamma_h^{m}} \eta_h &=0, \quad \forall \eta_h \in S_h(\Gamma_h^{m})^3,
    \end{align}
where \(\int_{\Gamma_h^{m}}^{(h)}\) denotes a mass-lumped surface integral, and \(n_h^{m}\) is the piecewisely defined normal vector on $\Gamma_h^{m}$. For any \(\eta_h\) that is tangential at every node, i.e.,
\(I_h(\eta_h\cdot n_h^{m})=0\) where \(I_h\) is the mass-lumped Lagrange interpolation operator defined in \eqref{eq:mass-lumped-int}, the BGN formulation implies
$\int_{\Gamma_h^{m}} \nabla_{\Gamma_h^{m}} X_h^{m+1} \cdot \nabla_{\Gamma_h^{m}} \eta_h = 0$. 
This identity shows that the one-step map $X_h^{m+1}:\Gamma_h^{m}\to\Gamma_h^{m+1}$ is discretely harmonic and therefore minimizes the tangential deformation from $\Gamma_h^{m}$ to $\Gamma_h^{m+1}$, thereby helping prevent severe distortion of the evolving surface triangulation. Beyond its favorable mesh-quality behavior, the BGN method also preserves a discrete area-dissipation law, thanks to \eqref{area-ineq}. Owing to this combination of mesh-quality control and energy stability, the BGN-type schemes are often more robust than Dziuk's original method in challenging simulations (including large deformations where Dziuk's method may fail),
and they have therefore been widely used for complex interfacial dynamics; see
\cite{fu2020arbitrary,Bao2021,Bao2023,Bai-Li-MCOM2025}.

Formally, as the time-step size \(\tau\to 0\), the BGN method can be interpreted as a purely spatial discretization of 
$$
(v\cdot n) n = \Delta_{\Gamma} \mathrm{id} .
$$ 
This relation prescribes only the normal velocity, and leaves the tangential motion undetermined. This lack of tangential information can translate, after spatial discretization, into an instability of the node motion: the computed velocity may no longer induce effective mesh redistribution. Consequently, when \(\tau\) is very small, the mesh quality can deteriorate in practice; see Figs.~\ref{fig:dum-BGN} and \ref{fig:Box118_BGN2}. 
Quantitative guidance for choosing \(\tau\) so as to reliably maintain good mesh quality is still lacking.

Motivated by this, Hu and Li~\cite{hu2022evolving} proposed a well-posed variant that can be interpreted as the sequential limit $\lim_{\tau\rightarrow0}\lim_{h\rightarrow0}$ of the BGN method. For mean curvature flow, they introduce a tangential motion by choosing, among all velocities satisfying the normal constraint
\(v\cdot n=-H\), the one that minimizes the deformation-rate energy
\begin{align}\label{DR-energy}
E(v)=\int_{\Gamma} |\nabla_{\Gamma} v|^2 .
\end{align}
This leads to the \emph{minimal-deformation-rate} (MDR) formulation
\begin{subequations}\label{eq:MDR-MCF}
\begin{align}
-\Delta_{\Gamma} v &= \kappa n, \label{eq:MDR-equ-v} \\
v \cdot n &= -H = \Delta_{\Gamma} \mathrm{id} \cdot n, \label{eq:MDR-equ-kappa} 
\end{align}
\end{subequations}
where \(\kappa\) is the Lagrange multiplier enforcing the normal-velocity constraint. The connection between the BGN and MDR viewpoints has been further exploited in
\cite{hu2022evolving,Bai-Li-MCOM2025,gao2026energy}. 

A direct discretization of \eqref{eq:MDR-MCF} reads as follows: find \((v_h^{m+1},\kappa_h^{m+1})\in S_h(\Gamma_h^{m})^3\times S_h(\Gamma_h^{m})\), which defines the one-step flow map $X_h^{m+1}=\mathrm{id}+\tau v_h^{m+1}:\Gamma_h^{m}\to \Gamma_h^{m+1}$, such that 
\begin{subequations}\label{MDR-weak}
    \begin{align}
        \int_{\Gamma_h^{m}} \nabla_{\Gamma_h^{m}} v_h^{m+1} \cdot \nabla_{\Gamma_h^{m}} \eta_h &=\int_{\Gamma_h^{m}} \kappa_h^{m+1}\hat n_h^{m} \cdot \eta_h , \\
        \int_{\Gamma_h^{m}} (v_h^{m+1} \cdot \hat n_h^{m})\,\phi_h + 
        \int_{\Gamma_h^{m}} \nabla_{\Gamma_h^{m}} X_h^{m+1} \cdot \nabla_{\Gamma_h^{m}} (\phi_h \hat n_h^{m}) &=0 ,
    \end{align}
\end{subequations}
for all $(\eta_h ,\phi_h) \in S_h(\Gamma_h^{m})^3 \times S_h(\Gamma_h^{m})$, where $\hat n_h^{m}\in S_h(\Gamma_h^{m})^3$ is the $L^2$-orthogonal projection of the piecewisely defined normal vector onto the finite element space. The resulting scheme maintains good mesh quality for both large and small time-step sizes, and its convergence was proved in \cite{huang2026+MDR}. However, it does not preserve the energy stability and therefore is less robust when resolving pinch-off singularities; see Fig.~\ref{fig:Box118-MDR}. This motivates the development of \emph{linearly implicit, energy-stable} discretizations that can reliably capture singularity formation (as in BGN-type schemes) while retaining the mesh-quality advantages of the MDR approach. 



Beyond the MDR approach, several alternative strategies have been proposed to introduce an artificial tangential motion (primarily aimed at improving mesh quality). Elliott and Fritz \cite{elliott2016algorithms,elliott2017approximations} constructed a tangential
redistribution via a DeTurck-type reparametrization, which was later extended to torus-type surfaces in \cite{mierswa2020error}. In a related spirit, \cite{GaoLiTang2026} generated tangential motion through a harmonic map heat flow and proved convergence of the numerical scheme (for a prescribed external velocity field). Duan and Li \cite{duan2024new} proposed a minimal-deformation (MD) formulation by minimizing the deformation energy of the flow map
$X(\cdot,t):\Gamma(0)\to\Gamma(t)$. These works are mainly concerned with designing tangential velocities that yield better mesh redistribution, while discrete energy (area) stability is typically not addressed. A structure-preserving MD variant was subsequently proposed in \cite{gao-li-2025} via 
introducing a scalar Lagrange multiplier; however, the resulting scheme is
nonlinearly implicit.

In this paper, we construct a new family of continuous formulations, called dual formulations, for mean curvature flow, surface diffusion, and solid-state dewetting. While equivalent to the classical models at the continuous level, these reformulations are tailored for the design of parametric FEMs that are simultaneously \emph{linearly implicit}, \emph{energy stable}, and \emph{mesh-quality preserving} for a broad class of artificial tangential motions, including the MDR tangential motion discussed above.
%
%

For mean curvature flow equipped with the MDR tangential motion, the corresponding dual formulation is a coupled system for \((v,\kappa,\lambda,H)\):
\begin{subequations}\label{eq:dual-MDR-MCF}
\begin{align}
-\Delta_{\Gamma} v &= \kappa n, \label{eq:dual-MDR-equ-v}\\
v \cdot n &= -H, \label{eq:dual-MDR-equ-kappa}\\
\Delta_{\Gamma} \lambda &= Hn + \Delta_{\Gamma} \mathrm{id}, \label{eq:dual-MDR-equ-lambda}\\
\lambda \cdot n &= 0 , \label{eq:dual-MDR-equ-H}
\end{align}
\end{subequations}
where \(v\) and \(\lambda\) are vector-valued unknowns, while \(\kappa\) and \(H\) are scalar-valued. \emph{At the continuous level, \eqref{eq:dual-MDR-MCF} is equivalent to the MDR formulation \eqref{eq:MDR-MCF}} in the following sense (shown in the next section): a quadruple \((v,\kappa,\lambda,H)\) is a solution of \eqref{eq:dual-MDR-MCF} if and only if \(\lambda\equiv 0\), \(H=-\Delta_\Gamma \mathrm{id}\cdot n\), and \((v,\kappa)\) is a solution of \eqref{eq:MDR-MCF}.
The auxiliary unknown \(\lambda\) is dual to \(v\) and is therefore referred to as a dual multiplier. Its role is not to alter the continuous dynamics, but to ensure that a \emph{linearly implicit discretization of \eqref{eq:dual-MDR-MCF} can incorporate the MDR tangential motion while preserving the underlying energy-stability structure.}

The dual formulation is not limited to mean curvature flow or to the MDR tangential motion. Rather, it provides a \emph{general framework} for constructing \emph{linearly implicit} and \emph{energy-stable} schemes for curvature-driven surface evolution while allowing for a \emph{broad class of artificial tangential motions}. In Section~2, we present the dual formulations and their discretizations for closed-surface evolution in mean curvature flow and surface diffusion. For clarity, we use the MDR tangential motion as a running example in the presentation of the dual formulations and their discretizations, and we include a subsection discussing extensions to other choices of tangential motion. Numerical experiments are presented to illustrate that the proposed discretizations can preserve energy stability and maintain good mesh quality simultaneously. In Section~3, we extend the approach to open surfaces with moving contact lines and report numerical results on benchmark problems that demonstrate the robustness of the approach in maintaining high-quality meshes. 



\section{Dual formulations and discretizations for closed-surface evolution}

In this section, we present the dual formulations and their discretizations for closed-surface evolution in mean curvature flow and surface diffusion. 

\subsection{Dual-MDR formulations}


The dual MDR formulation of mean curvature flow has been shown in \eqref{eq:dual-MDR-MCF}. Its equivalence to the original MDR formulation \eqref{eq:MDR-MCF} can be seen from the following arguments. 

%
Clearly, if \((v,\kappa)\) solves \eqref{eq:MDR-MCF}, then with \(H=-\Delta_{\Gamma}\mathrm{id}\cdot n\) (the mean curvature of \(\Gamma\)), the quadruple \((v,H,\lambda\equiv 0,\kappa)\) satisfies \eqref{eq:dual-MDR-MCF}.

Conversely, let \((v,H,\lambda,\kappa)\) be a solution of \eqref{eq:dual-MDR-MCF}. Testing \eqref{eq:dual-MDR-equ-lambda} with an arbitrary tangential vector field \(\eta\) and integrating by parts yields
$$
-\int_{\Gamma} \nabla_{\Gamma}\lambda \cdot \nabla_{\Gamma}\eta
= \int_{\Gamma} \bigl(Hn+\Delta_{\Gamma}\mathrm{id}\bigr)\cdot \eta 
= 0 ,
$$
where the last equality is because \(Hn\) and \(\Delta_{\Gamma}\mathrm{id}\) are both normal vector fields, whereas \(\eta\) is tangential. 
Moreover, \eqref{eq:dual-MDR-equ-H} implies that \(\lambda\) is tangential, so choosing \(\eta=\lambda\) gives
\[
\int_{\Gamma} \lvert \nabla_{\Gamma}\lambda\rvert^2 = 0.
\]
By the Poincar\'e inequality for tangential vector fields (see, e.g., \cite[Lemma 2.2]{hansbo2020analysis}), i.e.,
$$
\int_{\Gamma} \lvert \lambda\rvert^2
\le C_{\Gamma} \int_{\Gamma} \lvert \nabla_{\Gamma}\lambda\rvert^2 
\quad\mbox{(where $C_\Gamma$ is some constant depending on $\Gamma$)},
$$
it follows that \(\lambda\equiv 0\). Substituting \(\lambda\equiv 0\) into \eqref{eq:dual-MDR-equ-lambda} shows that \(H\) coincides with the mean curvature of \(\Gamma\). 
Consequently, \((v,\kappa)\) solves \eqref{eq:MDR-MCF}.

Analogously to the MDR formulation of mean curvature flow in \eqref{eq:MDR-MCF}, the MDR formulation of surface diffusion is given by
\begin{subequations}\label{eq:MDR-SD}
\begin{align}
-\Delta_{\Gamma} v &= \kappa n, \label{eq:MDR-SD-v}\\
v \cdot n &= \Delta_{\Gamma} H, \label{eq:MDR-SD-kappa}\\
H &= -\Delta_{\Gamma} \mathrm{id} \cdot n. \label{eq:MDR-SD-H}
\end{align}
\end{subequations}
Its dual formulation is the following coupled system for \((v,\kappa,\lambda,H)\):
\begin{subequations}\label{eq:dual-MDR-SD}
\begin{align}
-\Delta_{\Gamma} v &= \kappa n, \label{eq:dual-MDR-SD-v}\\
v \cdot n &= \Delta_{\Gamma} H, \label{eq:dual-MDR-SD-kappa}\\
\Delta_{\Gamma} \lambda &= Hn + \Delta_{\Gamma} \mathrm{id}, \label{eq:dual-MDR-SD-lambda}\\
\lambda \cdot n &= 0. \label{eq:dual-MDR-SD-H}
\end{align}
\end{subequations}
The equivalence of \eqref{eq:MDR-SD} and \eqref{eq:dual-MDR-SD} can be shown similarly. 


We now present the weak forms of the dual-MDR formulations \eqref{eq:dual-MDR-MCF} and \eqref{eq:dual-MDR-SD}, corresponding to mean curvature flow and surface diffusion, respectively. \medskip

\noindent\textit{Mean curvature flow:}
Find $(v, H, \lambda, \kappa) \in H^1(\Gamma)^3 \times H^1(\Gamma) \times H^1(\Gamma)^3 \times H^1(\Gamma)$ such that
\begin{subequations}\label{eq:mean curvature flow}
\begin{align}
\label{eq:mean curvature flow-H}
\int_{\Gamma} -H n \cdot w + \int_{\Gamma} \nabla_{\Gamma} \mathrm{id} \cdot \nabla_{\Gamma} w
&= \int_{\Gamma} \nabla_{\Gamma} \lambda \cdot \nabla_{\Gamma} w,\\
\label{eq:mean curvature flow-v}
\int_{\Gamma} (v \cdot n)\,\phi + \int_{\Gamma} H\,\phi
&= 0,\\
\label{eq:mean curvature flow-kappa}
\int_{\Gamma} \nabla_{\Gamma} v \cdot \nabla_{\Gamma} \eta
&= \int_{\Gamma} \kappa n \cdot \eta,\\
\label{eq:mean curvature flow-lambda}
\int_{\Gamma} \lambda \cdot n\,\varphi
&= 0,
\end{align}
\end{subequations}
for all $(w,\phi,\eta,\varphi)\in H^1(\Gamma)^3 \times H^1(\Gamma) \times H^1(\Gamma)^3 \times H^1(\Gamma)$.

\medskip
\noindent\textit{Surface diffusion:}
Find $(v, H, \lambda, \kappa) \in H^1(\Gamma)^3 \times H^1(\Gamma) \times H^1(\Gamma)^3 \times H^1(\Gamma)$ such that
\begin{subequations}\label{eq:surface diffusion}
\begin{align}
\label{eq:surface diffusion-H}
\int_{\Gamma} -H n \cdot w + \int_{\Gamma} \nabla_{\Gamma} \mathrm{id} \cdot \nabla_{\Gamma} w
&= \int_{\Gamma} \nabla_{\Gamma} \lambda \cdot \nabla_{\Gamma} w,\\
\label{eq:surface diffusion-v}
\int_{\Gamma} (v \cdot n)\,\phi + \int_{\Gamma} \nabla_{\Gamma} H \cdot \nabla_{\Gamma} \phi
&= 0,\\
\label{eq:surface diffusion-kappa}
\int_{\Gamma} \nabla_{\Gamma} v \cdot \nabla_{\Gamma} \eta
&= \int_{\Gamma} \kappa n \cdot \eta,\\
\label{eq:surface diffusion-lambda}
\int_{\Gamma} \lambda \cdot n\,\varphi
&= 0,
\end{align}
\end{subequations}
for all $(w,\phi,\eta,\varphi)\in H^1(\Gamma)^3 \times H^1(\Gamma) \times H^1(\Gamma)^3 \times H^1(\Gamma)$.


In the next subsection we show that the linearly implicit discretizations of \eqref{eq:mean curvature flow} and \eqref{eq:surface diffusion} preserve the energy-stability property automatically. This stands in contrast to linearly implicit discretizations of the original MDR formulations \eqref{eq:MDR-MCF} and \eqref{eq:MDR-SD}, for which energy stability is not guaranteed.

\subsection{Linearly implicit dual-MDR schemes for closed surfaces}

Let \(\Gamma_h^0\) be a triangulated surface in \(\mathbb{R}^3\) approximating the initial surface \(\Gamma^0\), whose elements are images of a reference triangle under affine maps. Let \(t_m=m\tau\), \(m=0,1,\ldots,N\), be a uniform partition of \([0,T]\), where \(\tau>0\) denotes the time-step size. We approximate the evolving surface $\Gamma(t_m)$ by a triangulated surface $\Gamma_h^m = \bigcup_{K \in \mathcal{K}_h^m} K$, where $\mathcal{K}_h^m$ denotes the set of triangles on $\Gamma_h^m$. 
The finite element space over \(\Gamma_h^{m}\) is defined by
\[
S_h(\Gamma_h^{m})=\left\{ v_h \in C^0(\Gamma_h^{m}) :  v_h \big|_{K}\ \text{is affine for all }K\in\mathcal{K}_h^{m}\right\}.
\]
The vector-valued finite element space is denoted by \(S_h(\Gamma_h^{m})^3\). 

For any triangle \(K\subset \Gamma_h^m\), we write \(K=\triangle(q_{\scriptscriptstyle K, 0},q_{\scriptscriptstyle  K,1},q_{\scriptscriptstyle  K,2})\), where the vertices \(q_{\scriptscriptstyle  K,0},q_{\scriptscriptstyle  K,1},q_{\scriptscriptstyle  K,2}\) are ordered counterclockwise as seen from the exterior of \(\Gamma_h^m\). The mass-lumped inner product on \(\Gamma_h^m\) is defined as 
\[
\int_{\Gamma_h^m}^{(h)} u \cdot w  := \sum_{K\in \mathcal{K}_h^m}\frac{|K|}{3}
\sum_{j=0}^2 \bigl((u\cdot w)|_K(q_{\scriptscriptstyle  K,j})\bigr),
\]
where \(|K|\) denotes the area of \(K\). For any piecewise-defined function \(f\) on \(\Gamma_h^m\) (possibly discontinuous across element interfaces), we define the mass-lumped Lagrange interpolant \(I_h^m f\in S_h(\Gamma_h^{m})\) via
\[
\int_{\Gamma_h^m}^{(h)} I_h^m f \cdot w = \int_{\Gamma_h^m}^{(h)}  f \cdot  w,\qquad \forall\, w\in S_h(\Gamma_h^{m}).
\]
It is straightforward to verify that this is equivalent to the nodal characterization
\begin{equation}\label{eq:mass-lumped-int}
  (I_h^m f)(q) =\frac{\sum_{K\ni q} f|_K(q) \cdot |K|}{\sum_{K\ni q}|K|},
  \qquad \text{for every vertex } q\in \Gamma_h^m.
\end{equation}

Based on the weak formulations \eqref{eq:mean curvature flow} and \eqref{eq:surface diffusion}, we propose the following \emph{dual-MDR scheme} for mean curvature flow and surface diffusion, where $n_h^m$ denotes the piecewise-constant unit normal vector on $\Gamma_h^m$.

\medskip
\noindent\textit{Mean curvature flow:}
Find $( v_h^{m+1},  H_h^{m+1}, \lambda_h^{m+1}, \kappa_h^{m+1}) \in S_h(\Gamma_h^{m})^3 \times S_h(\Gamma_h^{m}) \times S_h(\Gamma_h^{m})^3 \times S_h(\Gamma_h^{m})$ such that
\begin{subequations}\label{eq:num-mean curvature flow}
\begin{align}
\label{eq:num-mean curvature flow-H}
&-\int_{\Gamma_h^m}^{(h)}  H_h^{m+1} n_h^{m} \cdot w_h + \int_{\Gamma_h^m} \nabla_{\Gamma_h^m} ({\rm id} + \tau  v_h^{m+1}) \cdot \nabla_{\Gamma_h^m} w_h
= \int_{\Gamma_h^m} \nabla_{\Gamma_h^m} \lambda_h^{m+1} \cdot \nabla_{\Gamma_h^m} w_h, \\
  \label{eq:num-mean curvature flow-v}
&\int_{\Gamma_h^{m}}^{(h)} ( v_h^{m+1}\cdot n_h^{m}) \phi_h + \int_{\Gamma_h^{m}}^{(h)}  H_h^{m+1}  \phi_h 
  =0, \\
\label{eq:num-mean curvature flow-kappa}
&\int_{\Gamma_h^m} \nabla_{\Gamma_h^m}  v_h^{m+1} \cdot \nabla_{\Gamma_h^m} \eta_h = \int_{\Gamma_h^m}^{(h)} \kappa_h^{m+1} n_h^{m} \cdot \eta_h, \\
\label{eq:num-mean curvature flow-lambda}
&\int_{\Gamma_h^m}^{(h)} \lambda_h^{m+1} \cdot n_h^{m} \varphi_h = 0,
\end{align}
\end{subequations}
holds for all $(w_h, \phi_h, \eta_h, \varphi_h) \in S_h(\Gamma_h^{m})^3 \times S_h(\Gamma_h^{m})  \times S_h(\Gamma_h^{m})^3 \times S_h(\Gamma_h^{m}) $.

\medskip
\noindent\textit{Surface diffusion:}
Find $( v_h^{m+1},  H_h^{m+1}, \lambda_h^{m+1}, \kappa_h^{m+1}) \in S_h(\Gamma_h^{m})^3 \times S_h(\Gamma_h^{m}) \times S_h(\Gamma_h^{m})^3 \times S_h(\Gamma_h^{m})$ such that
\begin{subequations}\label{eq:num-surface diffusion}
\begin{align}
\label{eq:num-surface diffusion-H}
&-\int_{\Gamma_h^{m}}^{(h)} H_h^{m+1} n_h^{m} \cdot w_h + \int_{\Gamma_h^{m}} \nabla_{\Gamma_h^{m}}({\rm id} + \tau  v_h^{m+1}) \cdot \nabla_{\Gamma_h^{m}} w_h
= \int_{\Gamma_h^{m}} \nabla_{\Gamma_h^{m}} \lambda_h^{m+1} \cdot \nabla_{\Gamma_h^{m}} w_h, \\
\label{eq:num-surface diffusion-v}
&\int_{\Gamma_h^{m}}^{(h)} ( v_h^{m+1}\cdot n_h^{m}) \phi_h + \int_{\Gamma_h^{m}}\nabla_{\Gamma_h^{m}}  H_h^{m+1} \cdot \nabla_{\Gamma_h^m} \phi_h 
=0, \\
\label{eq:num-surface diffusion-kappa}
& \int_{\Gamma_h^{m}} \nabla_{\Gamma_h^{m}}  v_h^{m+1} \cdot \nabla_{\Gamma_h^{m}} \eta_h = \int_{\Gamma_h^{m}}^{(h)} \kappa_h^{m+1} n_h^{m} \cdot \eta_h,\\
 \label{eq:num-surface diffusion-lambda}
&\int_{\Gamma_h^m}^{(h)} \lambda_h^{m+1} \cdot n_h^{m} \varphi_h = 0,
\end{align}
\end{subequations}
holds for all $(w_h, \phi_h, \eta_h, \varphi_h) \in S_h(\Gamma_h^{m})^3 \times S_h(\Gamma_h^{m})  \times S_h(\Gamma_h^{m})^3 \times S_h(\Gamma_h^{m}) $.\medskip

After solving \eqref{eq:num-mean curvature flow} or \eqref{eq:num-surface diffusion}, the discrete surface is updated by
\[
\Gamma_h^{m+1} := X_h^{m+1}(\Gamma_h^m), \quad \text{where} \quad
X_h^{m+1} = {{\rm id} } + \tau  v_h^{m+1}.
\]

For the well-posedness analysis, we further define the averaged normal vector \(\hat{n}_h^m = I_h^m n_h^m \in S_h(\Gamma_h^{m})^3\). According to \eqref{eq:mass-lumped-int}, it satisfies
\[
\hat{n}_h^m(q) = \frac{\sum_{K\ni q} n_h^m|_K \cdot |K|}{\sum_{K\ni q}|K|},
\qquad \text{for every vertex } q\in \Gamma_h^m.
\]
The well-posedness of \eqref{eq:num-mean curvature flow} and \eqref{eq:num-surface diffusion} can be established under the following mild assumptions (similarly as \cite{Bao2023,li2021energy,garcke2025stable}).

\begin{theorem}[Well-posedness of the numerical scheme]\label{lm:num-wellpose}
Assume that the discrete surface $\Gamma_h^m$ satisfies the following mild conditions:
  \begin{enumerate}[label={{\rm(A\arabic*)}}]
    \item The elements are nondegenerate, i.e., for each \(K\in \mathcal{K}_h^{m}\) it holds that
    \(|K|>0\).
    \item For each vertex \(q\) of \(\Gamma_h^m\), the averaged normal vector satisfies
    \(\hat{n}_h^m(q)\neq {0}\) and
    \[
    \dim\!\Bigl(\operatorname{span}\{\hat{n}_h^m(q)\,:\, q \text{ is a vertex of } \Gamma_h^m\}\Bigr)=3.
    \]
  \end{enumerate}
Then, each of the numerical schemes \eqref{eq:num-mean curvature flow} and \eqref{eq:num-surface diffusion} admits a unique solution
\[
( v_h^{m+1},  H_h^{m+1}, \lambda_h^{m+1}, \kappa_h^{m+1}) \in S_h(\Gamma_h^{m})^3 \times S_h(\Gamma_h^{m}) \times S_h(\Gamma_h^{m})^3 \times S_h(\Gamma_h^{m}).
\]
\end{theorem}

\begin{proof}
We consider the scheme \eqref{eq:num-mean curvature flow} for example. It suffices to show that the following homogeneous system admits only the zero solution:
\begin{subequations}\label{eq:homo-mean curvature flow}
\begin{align}
\label{eq:homo-mean curvature flow-v}
\int_{\Gamma_h^{m}}^{(h)} ( v_h^{m+1}\cdot n_h^{m}) \phi_h + \int_{\Gamma_h^{m}}^{(h)}  H_h^{m+1}  \phi_h
&=0, \\
\label{eq:homo-mean curvature flow-H}
-\int_{\Gamma_h^m}^{(h)}  H_h^{m+1} n_h^{m} \cdot w_h + \tau \int_{\Gamma_h^m} \nabla_{\Gamma_h^m}   v_h^{m+1} \cdot \nabla_{\Gamma_h^m} w_h
&= \int_{\Gamma_h^m} \nabla_{\Gamma_h^m} \lambda_h^{m+1} \cdot \nabla_{\Gamma_h^m} w_h,
\\
\label{eq:homo-mean curvature flow-kappa}
\int_{\Gamma_h^m} \nabla_{\Gamma_h^m}  v_h^{m+1} \cdot \nabla_{\Gamma_h^m} \eta_h
&= \int_{\Gamma_h^m}^{(h)} \kappa_h^{m+1} n_h^{m} \cdot \eta_h,
\\
\label{eq:homo-mean curvature flow-lambda}
\int_{\Gamma_h^m}^{(h)} \lambda_h^{m+1} \cdot n_h^{m} \varphi_h
&= 0,
\end{align}
\end{subequations}
holds for all $(w_h, \phi_h, \eta_h, \varphi_h) \in S_h(\Gamma_h^{m})^3 \times S_h(\Gamma_h^{m})  \times S_h(\Gamma_h^{m})^3 \times S_h(\Gamma_h^{m}) $.

By choosing the test functions \(\phi_h = I_h^m( v_h^{m+1} \cdot n_h^m)\), \(w_h =  v_h^{m+1}\), \(\eta_h = \lambda_h^{m+1}\), and \(\varphi_h = \kappa_h^{m+1}\), and summing up the resulting equations, we obtain
\[
\int_{\Gamma_h^m}^{(h)} I_h^m( v_h^{m+1} \cdot n_h^{m}) I_h^m( v_h^{m+1} \cdot n_h^{m})
+ \tau\int_{\Gamma_h^m} \nabla_{\Gamma_h^m}  v_h^{m+1} \cdot \nabla_{\Gamma_h^m}  v_h^{m+1}
= 0 .
\]
This implies
\[
\| \nabla_{\Gamma_h^m}  v_h^{m+1} \|_{L^2(\Gamma_h^m)} = 0
\quad \text{and} \quad
I_h^m( v_h^{m+1} \cdot n_h^{m}) = 0.
\]
By the nondegeneracy condition (A1), the first identity implies that
\( v_h^{m+1} = v_{\mathrm{const}}\) is a constant vector field. Combined with the second identity, this leads to
\[
0 = I_h^m( v_h^{m+1} \cdot n_h^{m})
= v_{\mathrm{const}} \cdot I_h^m(n_h^m)
= v_{\mathrm{const}} \cdot \hat{n}_h^m.
\]
Hence, by condition (A2), we conclude that \( v_h^{m+1} = v_{\mathrm{const}} = 0\).

Substituting \( v_h^{m+1} = 0\) into \eqref{eq:homo-mean curvature flow-v} immediately yields \( H_h^{m+1} = 0\). Next, taking \(w_h = \lambda_h^{m+1}\) in \eqref{eq:homo-mean curvature flow-H} and \(\varphi_h = I_h^m(n_h^m \cdot \lambda_h^{m+1})\) in \eqref{eq:homo-mean curvature flow-lambda}, we obtain
\[
\| \nabla_{\Gamma_h^m} \lambda_h^{m+1} \|_{L^2(\Gamma_h^m)} = 0
\quad \text{and} \quad
I_h^m(\lambda_h^{m+1} \cdot n_h^{m}) = 0.
\]
By the same argument as above, we conclude that \(\lambda_h^{m+1} = 0\).

Finally, taking \(\eta_h = I_h^m(\kappa_h^{m+1} n_h^m)\) in \eqref{eq:homo-mean curvature flow-kappa}, we have
\[
0 = \int_{\Gamma_h^m}^{(h)} \kappa_h^{m+1} n_h^{m} \cdot I_h^m(\kappa_h^{m+1} n_h^m)
= \int_{\Gamma_h^m}^{(h)} \bigl|I_h^m(\kappa_h^{m+1} n_h^m)\bigr|^2
= \int_{\Gamma_h^m}^{(h)} \bigl|I_h^m(\kappa_h^{m+1} \hat{n}_h^m)\bigr|^2,
\]
where, in the last step, we use \(I_h^m(\kappa_h^{m+1} n_h^m) = I_h^m(\kappa_h^{m+1} \hat{n}_h^m)\), which can be verified by comparing their nodal values via \eqref{eq:mass-lumped-int}. This implies that \(\kappa_h^{m+1} = 0\), since condition (A2) ensures that \(\hat{n}_h^m\) is nonzero at each vertex of \(\Gamma_h^m\).

The proof for scheme \eqref{eq:num-surface diffusion} follows the same argument as that for scheme \eqref{eq:num-mean curvature flow} and is therefore omitted.
\hfill\end{proof}

The following result shows that the proposed schemes \eqref{eq:num-mean curvature flow} and \eqref{eq:num-surface diffusion} are unconditionally energy stable.

\begin{theorem}\label{thm:area-decreasing-mass-mcf}
For the numerical solutions determined by \eqref{eq:num-mean curvature flow} and \eqref{eq:num-surface diffusion}, the discrete surface area is monotonically non-increasing in time, i.e., 
\[
|\Gamma_h^{m+1}| \le |\Gamma_h^{m}| .
\]
\end{theorem}

\begin{proof}
Choosing $(w_h, \phi_h, \eta_h, \varphi_h)=(v_h^{m+1},H_h^{m+1}, \lambda_h^{m+1},\kappa_h^{m+1})$ in \eqref{eq:num-mean curvature flow} and summing the resulting identities yields
\[
\int_{\Gamma_h^m}^{(h)} \bigl|H_h^{m+1}\bigr|^2
\;+\;
\int_{\Gamma_h^m} 
\nabla_{\Gamma_h^m}\bigl(\mathrm{id}+\tau v_h^{m+1}\bigr)
\cdot \nabla_{\Gamma_h^m} v_h^{m+1}
= 0.
\]
For \eqref{eq:num-surface diffusion}, the same choice of test functions gives
\[
\int_{\Gamma_h^m} \bigl|\nabla_{\Gamma_h^m} H_h^{m+1}\bigr|^2
\;+\;
\int_{\Gamma_h^m} 
\nabla_{\Gamma_h^m}\bigl(\mathrm{id}+\tau v_h^{m+1}\bigr)
\cdot\nabla_{\Gamma_h^m} v_h^{m+1}
= 0.
\]
Energy stability for both schemes follows from \(\mathrm{id}+\tau v_h^{m+1}=X_h^{m+1}\) and \eqref{area-ineq}.
\hfill\end{proof}

\subsection{Dual formulation for other artificial tangential motions}
\label{section:tangential}

The dual formulation is not restricted to the MDR choice; rather, it provides a general framework for constructing linearly implicit and energy-stable schemes for curvature flows equipped with other tangential motions. To this end, let \(a_{\Gamma}(\cdot,\cdot)\) be a symmetric bilinear form on vector fields over \(\Gamma\), which is positive definite when restricted to the tangential subspace. For example,
\begin{align}\label{sym-MDR}
a_{\Gamma}(v,v)
=\int_{\Gamma} \frac12\bigl|\nabla_{\Gamma}v\bigr|^2 
\quad\mbox{or}\quad
a_{\Gamma}(v,v)
=\int_{\Gamma} \frac12\bigl|\nabla_{\Gamma}v+(\nabla_{\Gamma}v)^{\mathrm T}\bigr|^2 .
\end{align}
One may then define a tangential velocity by minimizing the {\it generalized MDR energy} \(a_{\Gamma}(v,v)\) subject to the prescribed normal-velocity constraint. The corresponding dual formulation of mean curvature flow (in weak form) reads as follows.
\medskip

\noindent\emph{Generalized dual-MDR formulation of mean curvature flow:}
Find \((v,H,\lambda,\kappa)\in H^1(\Gamma)^3\times H^1(\Gamma)\times H^1(\Gamma)^3\times H^1(\Gamma)\) such that
\begin{subequations}\label{eq:sym-MCF}
\begin{align}
\label{eq:sym-mean curvature flow-H}
\hspace*{-1.2em}\int_{\Gamma} -Hn\cdot w
+\int_{\Gamma} \nabla_{\Gamma}\mathrm{id}\cdot\nabla_{\Gamma} w
&= a_{\Gamma}(\lambda,w),
&& \forall\, w\in H^1(\Gamma)^3, \\
\label{eq:sym-mean curvature flow-v}
\hspace*{-1.2em}\int_{\Gamma} (v\cdot n)\,\phi
+\int_{\Gamma} H\,\phi
&= 0,
&& \forall\, \phi\in H^1(\Gamma), \\
\label{eq:sym-mean curvature flow-kappa}
\hspace*{-1.2em}a_{\Gamma}(v,\eta)
&= \int_{\Gamma} \kappa n\cdot \eta,
&& \forall\, \eta\in H^1(\Gamma)^3, \\
\label{eq:sym-mean curvature flow-lambda}
\hspace*{-1.2em}\int_{\Gamma} (\lambda\cdot n)\,\varphi
&= 0,
&& \forall\, \varphi\in H^1(\Gamma).
\end{align}
\end{subequations}
A discretization of \eqref{eq:sym-MCF}, analogous to \eqref{eq:num-mean curvature flow}, yields a linearly implicit and energy-stable scheme for the tangential motion associated with \(a_{\Gamma}(\cdot,\cdot)\).

The corresponding dual formulation and discretization for surface diffusion with the tangential motion associated with \(a_{\Gamma}(\cdot,\cdot)\) are analogous and are therefore omitted.

\subsection{Numerical experiments for closed-surface evolution}\label{sec:numerical}

In this section we present numerical experiments to demonstrate the performance of the proposed dual-MDR schemes for closed-surface evolution in mean curvature flow and surface diffusion. We report errors and convergence rates for the proposed methods, and assess their effectiveness in improving mesh quality by comparing them with the MDR and BGN methods. To quantify the mesh quality of a polyhedral surface, we use the indicator
\[
\sigma_{\max} = \max_{K \subset \Gamma_h} \frac{h(K)}{r(K)},
\]
where $h(K)$ is the diameter of the circumcircle and $r(K)$ is the diameter of the largest inscribed circle of a triangle $K \subset \Gamma_h$.

\begin{example}[Convergence of the method for mean curvature flow]
\label{ex:mean curvature flow-rate}
\upshape
We test the accuracy and convergence rate of the scheme in \eqref{eq:num-mean curvature flow} for mean curvature flow by taking the unit sphere as the initial surface. For this setting, the analytical solution is known for $t<0.25$, and the evolving surface remains spherical with radius
\(
r(t)=\sqrt{1-4t}.
\)
At the final time $t=T$, the error of the computed surface obtained with time-step size $\tau$ and $N_p$ vertices is quantified by
\[
\mathrm{Error}(\tau, N_p)=\max_{j=1,\ldots,N_p}\Bigl|\;|x_j(T)|-r(T)\Bigr|,
\]
where $x_j(T) \in \mathbb{R}^3$ represents a vertex of the triangulated surface $\Gamma_h^N$, with $N = T/\tau$.
To test convergence with respect to the temporal discretization, we use a highly refined triangulation of the initial surface with \(N_p=5238\) vertices and \(N_T=10472\) triangles, so that spatial discretization errors are negligible. 
To test convergence with respect to the spatial discretization, we fix a very small time-step size \(\tau=10^{-5}\) to suppress temporal discretization errors. The resulting errors and observed convergence rates at \(T=0.1\) are reported in Fig.~\ref{fig:mean curvature flow-conv-space} and Fig.~\ref{fig:mean curvature flow-conv-time}. These results indicate that the temporal and spatial discretization errors are $\mathcal{O}(\tau)$ and $\mathcal{O}(N_p^{-1})$, respectively. 



  \begin{figure}[htbp]
\centering
\captionsetup[subfigure]{skip=2pt}

\begin{subfigure}[t]{0.42\textwidth}
  \centering
  \includegraphics[width=\linewidth]{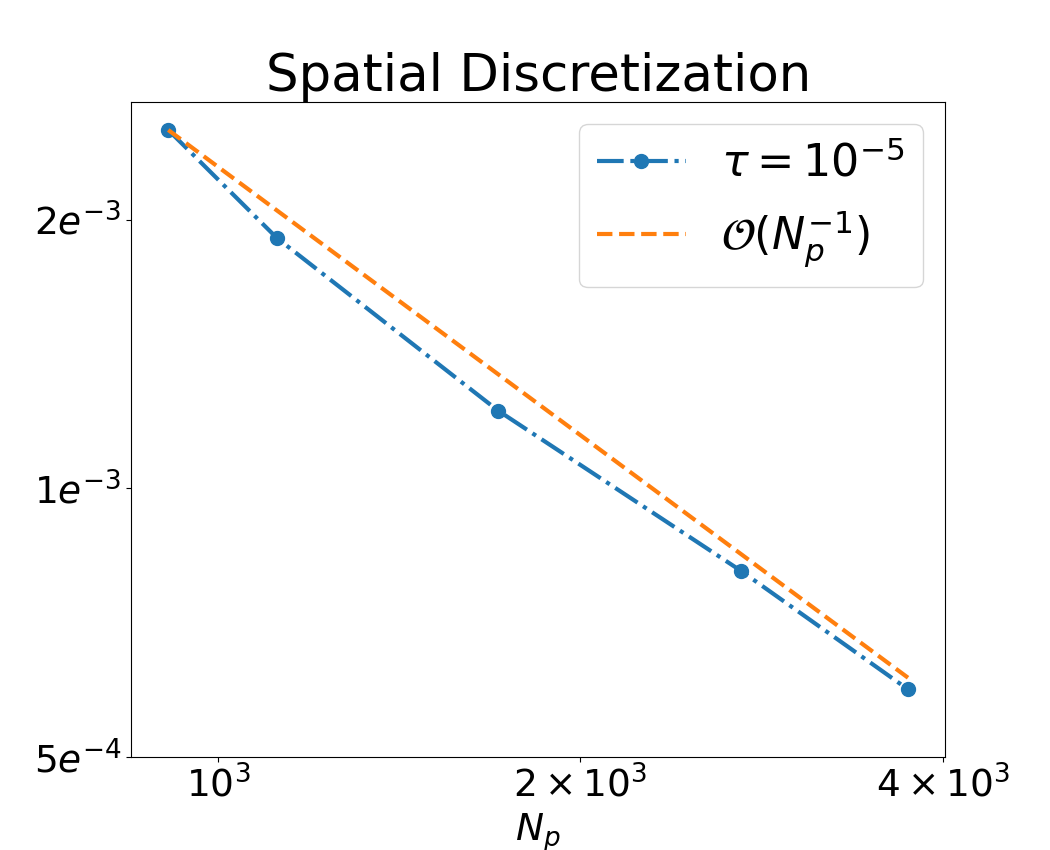}
  \caption{Error of spatial discretization}
  \label{fig:mean curvature flow-conv-space}
\end{subfigure}\qquad
\begin{subfigure}[t]{0.42\textwidth}
  \centering
  \includegraphics[width=\linewidth]{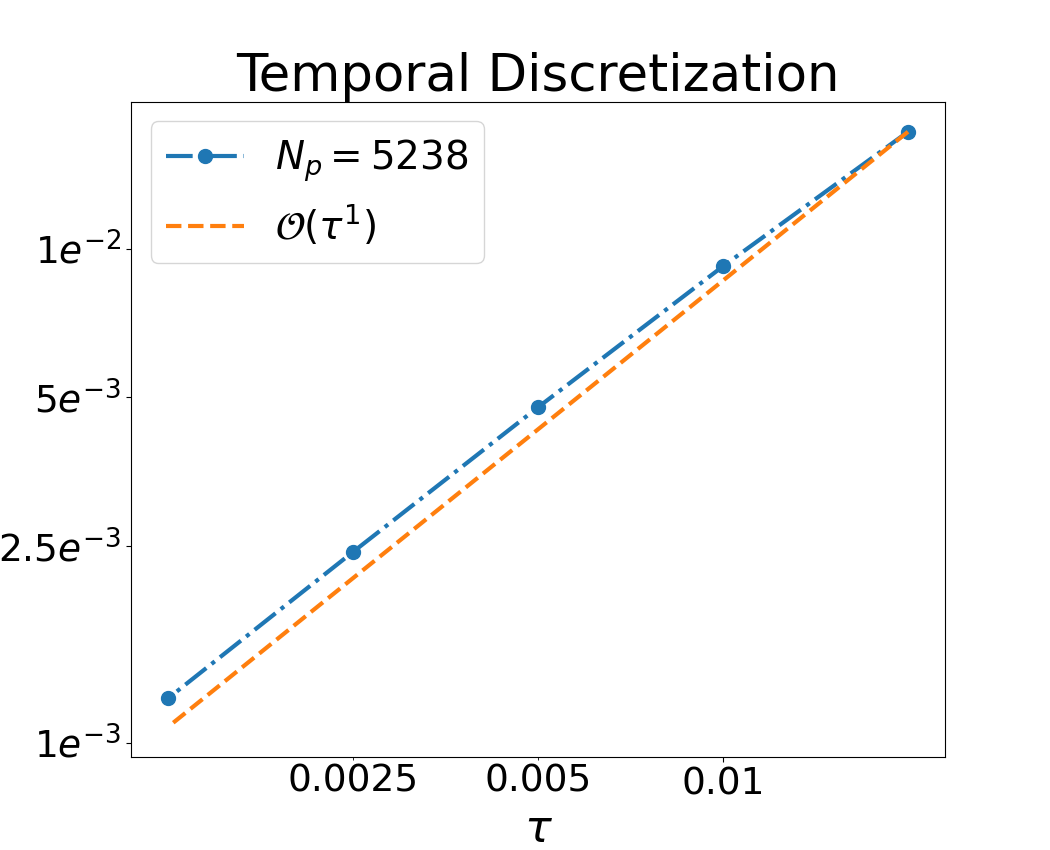}
  \caption{Error of temporal discretization}
  \label{fig:mean curvature flow-conv-time}
\end{subfigure}

\caption{Example~\ref{ex:mean curvature flow-rate}: convergence rates in space and time for scheme~\eqref{eq:num-mean curvature flow}.}
\label{fig:mean curvature flow-conv-space-time}
\end{figure}

\end{example}

\begin{example}[Mean curvature flow for dumbbell-shaped surface]\label{ex:mean curvature flow-dumbbell}\upshape
We consider a benchmark example of mean curvature flow with a dumbbell-shaped initial surface given by the following parameterization:
\begin{equation}
x=\left(\begin{array}{c}
\cos \varphi \\
\big(0.6 \cos ^2 \varphi+0.4\big) \cos \theta \sin \varphi \\
\big(0.6 \cos ^2 \varphi+0.4\big) \sin \theta \sin \varphi
\end{array}\right), \quad \theta \in[0,2 \pi), \quad \varphi \in[0, \pi].
\end{equation}

\begin{figure}[htbp!]
    \centering
    \begin{subfigure}[b]{0.5\textwidth}
        \centering
        \includegraphics[width=0.6\textwidth]{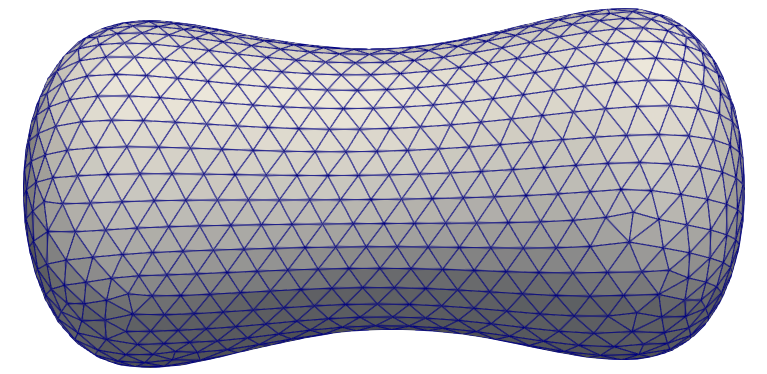}
        \caption{Initial surface}
    \end{subfigure}

    \vspace{-5pt}

    \begin{subfigure}[b]{0.3\textwidth}
        \centering
        \includegraphics[width=0.6\textwidth]{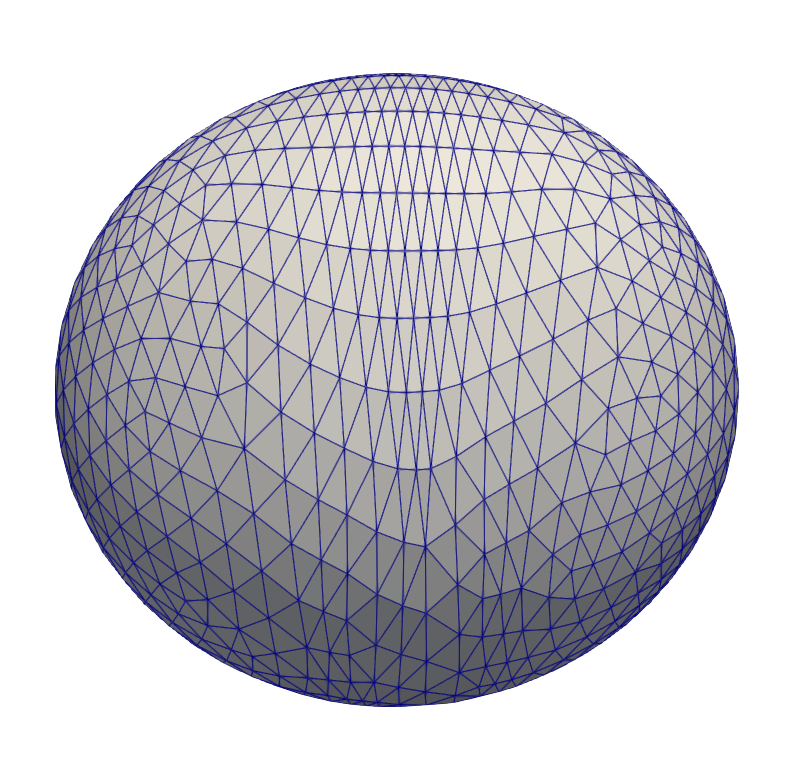}
        \caption{BGN with $\tau = 1 \times 10^{-4}$ at $t = 0.0911$}
        \label{fig:dum-BGN2}
    \end{subfigure}
\begin{subfigure}[b]{0.3\textwidth}
    \centering
        \includegraphics[width=0.63\textwidth]{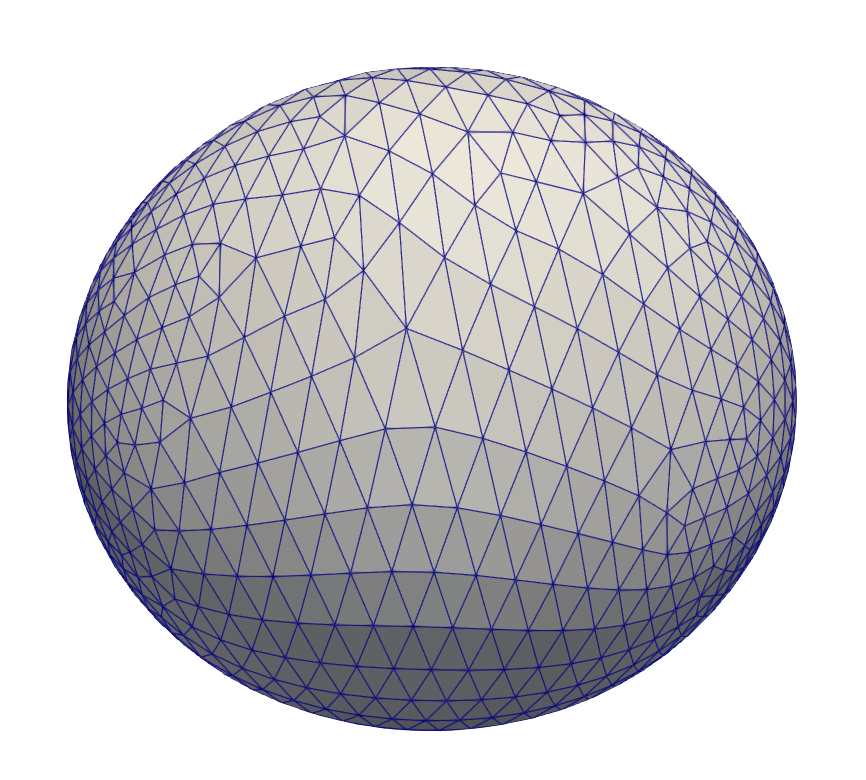}
    \caption{MDR with $\tau = 1 \times 10^{-4}$ at $t = 0.0913$}
    \label{fig:dum-MDR2}
\end{subfigure}
 \begin{subfigure}[b]{0.3\textwidth}
        \centering
        \includegraphics[width=0.6\textwidth]{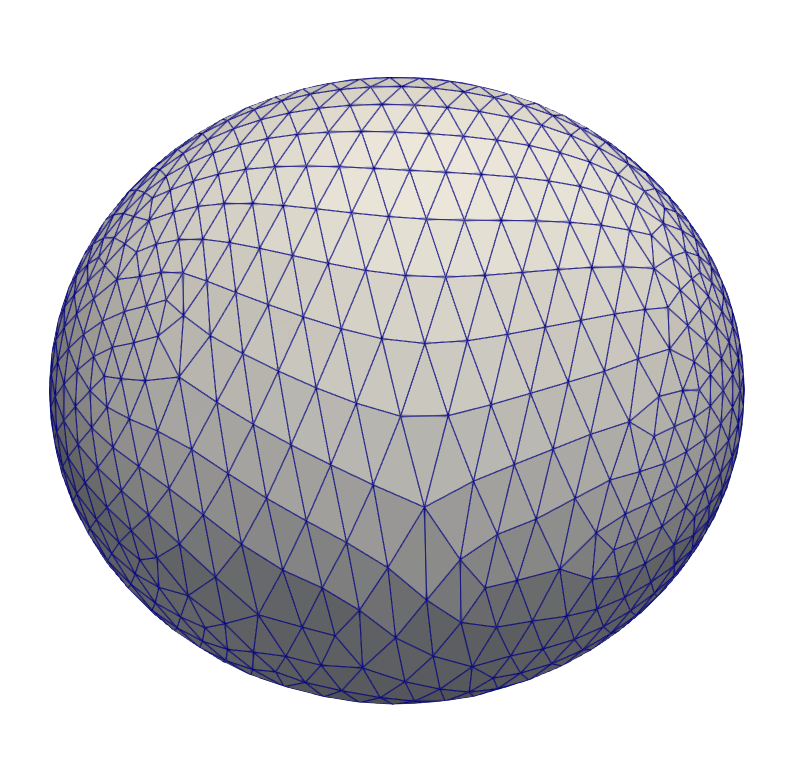}
        \caption{Dual-MDR with $\tau = 1 \times 10^{-4}$ at $t = 0.0911$}
        \label{fig:dum-MDR-dual2}
    \end{subfigure}

    \vspace{-5pt}
    \begin{subfigure}[b]{0.3\textwidth}
        \centering
        \includegraphics[width=0.6\textwidth]{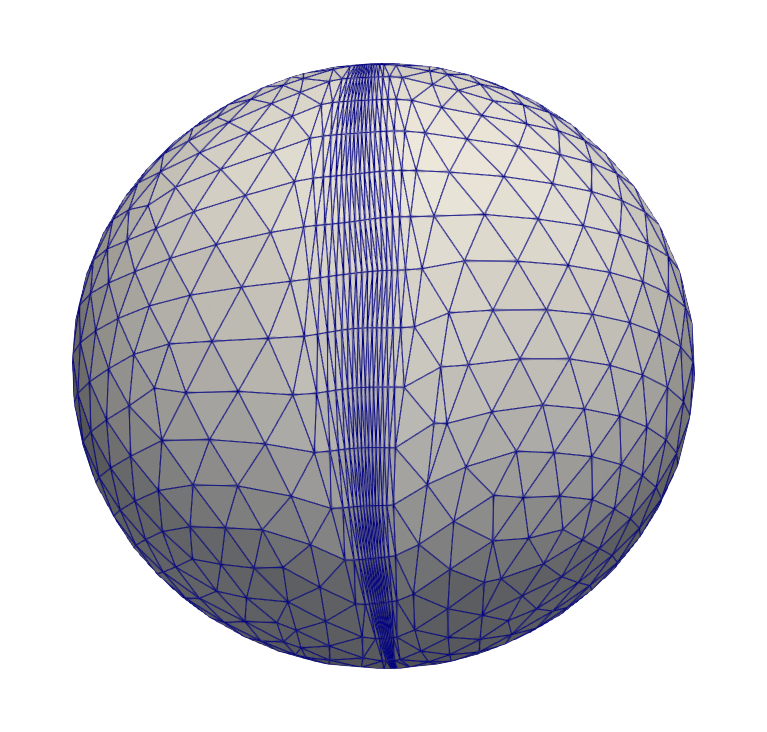}
        \caption{BGN with $\tau = 2.5 \times 10^{-5}$ at $t = 0.090625$}
        \label{fig:dum-BGN}
    \end{subfigure}
\begin{subfigure}[b]{0.3\textwidth}
    \centering
        \includegraphics[width=0.63\textwidth]{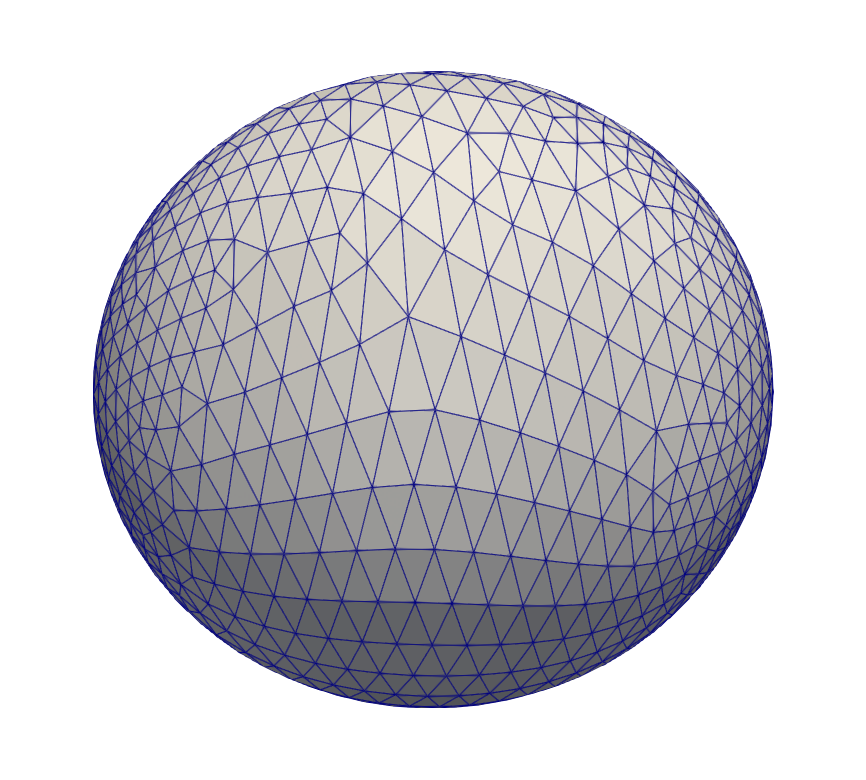}
    \caption{MDR with $\tau = 2.5 \times 10^{-5}$ at $t = 0.090775$}
    \label{fig:dum-MDR}
\end{subfigure}
 \begin{subfigure}[b]{0.3\textwidth}
        \centering
        \includegraphics[width=0.6\textwidth]{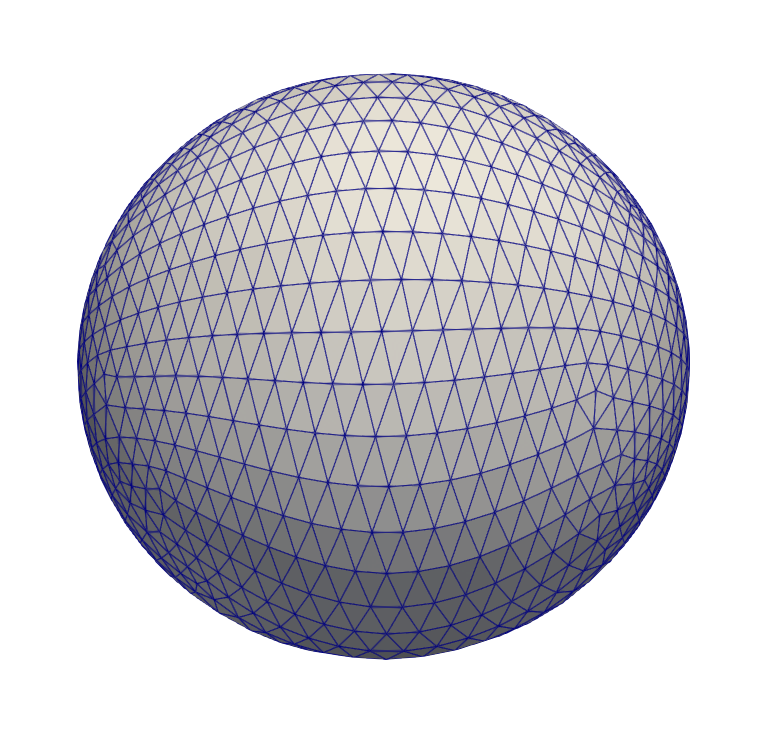}
        \caption{Dual-MDR with $\tau = 2.5 \times 10^{-5}$ at $t = 0.090575$}
        \label{fig:dum-MDR-dual}
    \end{subfigure}

\caption{Comparison of surface evolution under mean curvature flow in Example~\ref{ex:mean curvature flow-dumbbell}.}
\label{fig:dumPic}
\end{figure}
It has been shown in \cite{elliott2017approximations} that this dumbbell-shaped surface evolves toward a spherical shape under mean curvature flow, while its surface area decreases by several orders of magnitude before the surface ultimately shrinks to a point. Such a dramatic reduction in area makes this example particularly challenging from the numerical point of view, since Dziuk’s method, as well as other approaches that do not incorporate artificial tangential motion, may suffer from mesh distortion and inaccurate geometric evolution, thereby preventing the surface from approaching the correct spherical shape.

In our numerical experiments, the initial surface is triangulated into 2152 triangles and 1078 vertices. We compare the BGN scheme, the MDR scheme, and our dual-MDR scheme \eqref{eq:num-mean curvature flow} for different time-step sizes in Fig.~\ref{fig:dumPic}. All three schemes approach the correct spherical shape as the solution develops singularities. However, when a small time-step size is used, the BGN method exhibits noticeable mesh distortion (see Fig.~\ref{fig:dum-BGN}), whereas our method performs robustly across different time-step sizes. Furthermore, Fig.~\ref{fig:dum-MQ} compares the mesh quality of the different methods for $\tau = 1\times 10^{-5}$ and confirms these observations. Finally, Fig.~\ref{fig:dum-Area} and Fig.~\ref{fig:dum-dArea} illustrate the energy stability of our method, showing that it preserves the energy-dissipation structure.




\begin{figure}[htbp]
\centering
\captionsetup[subfigure]{skip=2pt} 

\begin{subfigure}[t]{0.333\textwidth}
  \centering
  \includegraphics[width=\linewidth]{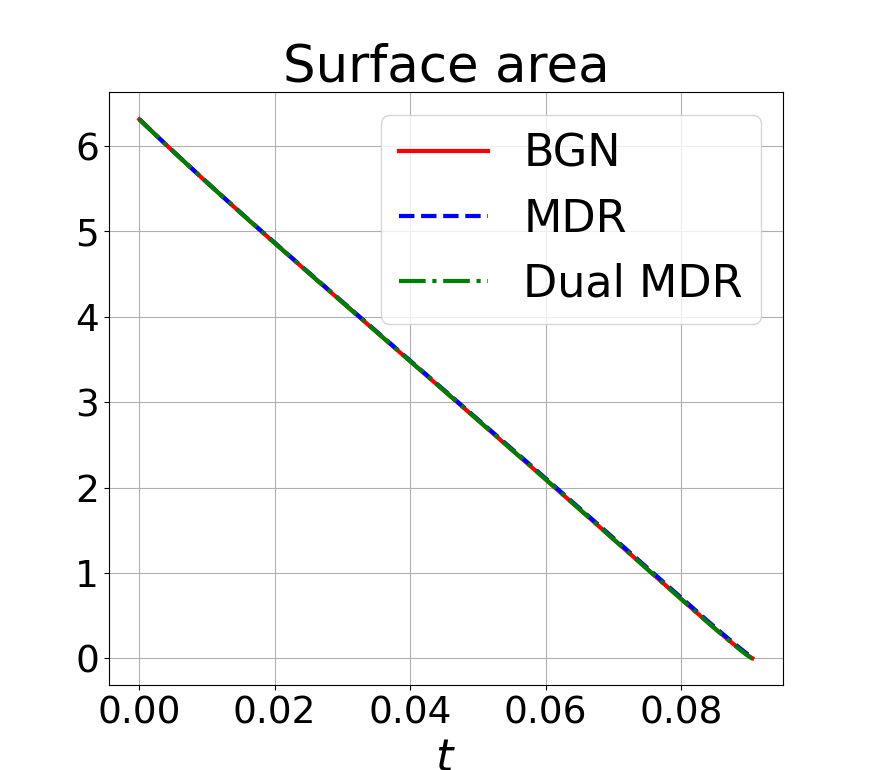}
  \caption{Surface area: $|\Gamma_h^{m}|$}
  \label{fig:dum-Area}
\end{subfigure}\hfill
\begin{subfigure}[t]{0.333\textwidth}
  \centering
  \includegraphics[width=\linewidth]{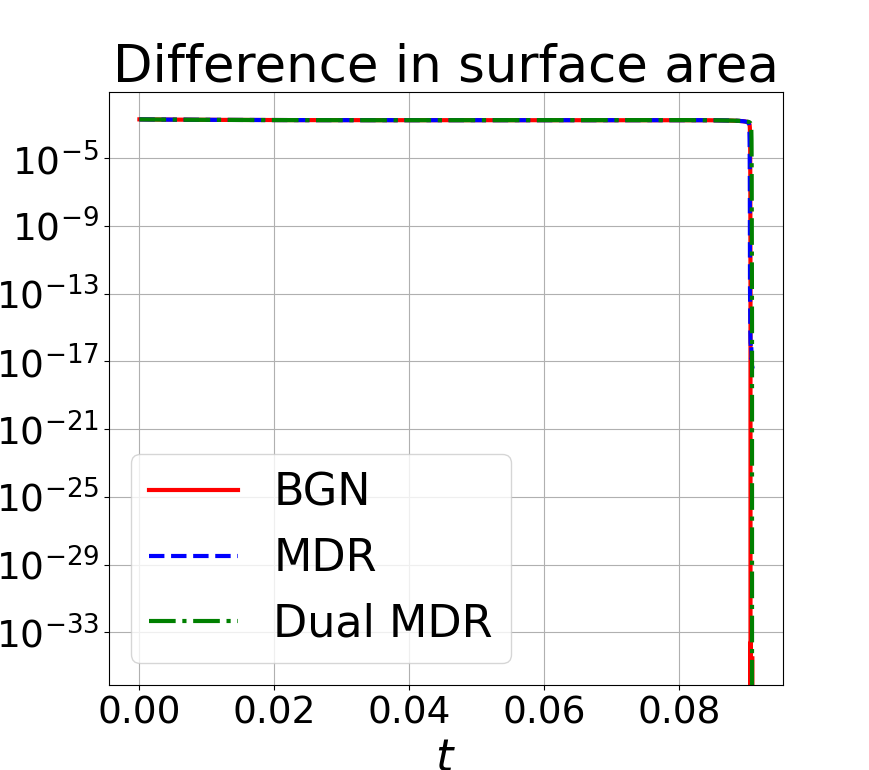}
  \caption{$|\Gamma_h^{m}| - |\Gamma_h^{m+1}|$}
  \label{fig:dum-dArea}
\end{subfigure}\hfill
\begin{subfigure}[t]{0.333\textwidth}
  \centering
  \includegraphics[width=\linewidth]{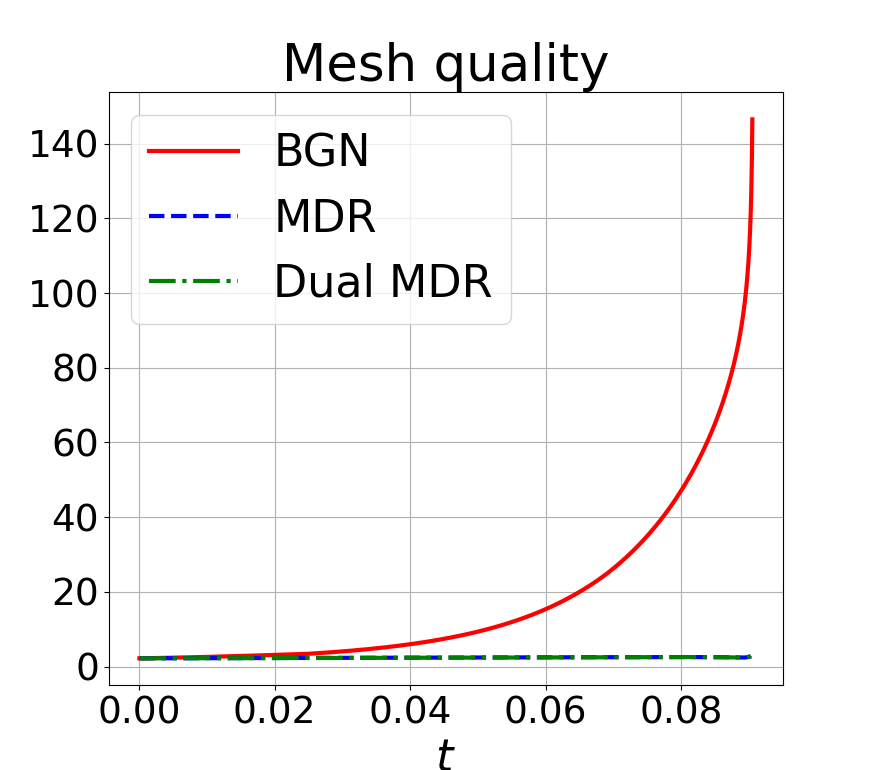}
  \caption{Mesh quality}
  \label{fig:dum-MQ}
\end{subfigure}

\caption{Comparison of the surface-area decay and mesh quality in Example~\ref{ex:mean curvature flow-dumbbell}.}
\label{fig:dum-MQ-Area}
\end{figure}

\end{example}

\begin{example}[Surface diffusion of a 1:1:8 cuboid]\label{ex:surface diffusion-Box118}
  \upshape
In this example, we study the surface diffusion of a $1\!:\!1\!:\!8$ cuboid that develops singularities. The initial surface is nonsmooth, with sharp edges and corners. Under surface diffusion, these sharp features gradually smooth and round off, after which the midsection progressively thins, eventually leading to pinch-off and separation into two components. The absence of tangential motion may cause triangles near edges to cross and fold, resulting in severe mesh distortion. Moreover, the formation of singularities poses significant challenges for numerical simulation, and without energy stability, the singular dynamics may not be captured accurately.

Figure~\ref{fig:Box118-Pic} compares the numerical results produced by the BGN method, the MDR scheme, and the dual-MDR scheme \eqref{eq:num-surface diffusion}. 
We observe that the MDR scheme becomes unstable near pinch-off due to the lack of energy stability (see Fig.~\ref{fig:Box118-MDR} and Fig.~\ref{fig:Box118-MDR2}). 
In contrast, both the BGN method and the dual-MDR scheme capture the pinch-off singularity. 

More specifically, the BGN method resolves the pinch-off event sharply (see Fig.~\ref{fig:Box118-BGN} and Fig.~\ref{fig:Box118_BGN2}), producing a discrete surface that nearly splits into two components connected only by a very thin neck. The dual-MDR scheme also captures pinch-off correctly and, moreover, maintains good mesh quality for both large and small time-step sizes (see Fig.~\ref{fig:Box118-MDR-dual} and Fig.~\ref{fig:Box118-MDR-dual2}). 
By comparison, the BGN method exhibits noticeable mesh distortion for the smaller time-step size (see Fig.~\ref{fig:Box118_BGN2}).

\begin{figure}[htbp]
    \centering
    \begin{subfigure}[b]{0.6\textwidth}
        \centering
        \includegraphics[width=0.6\textwidth]{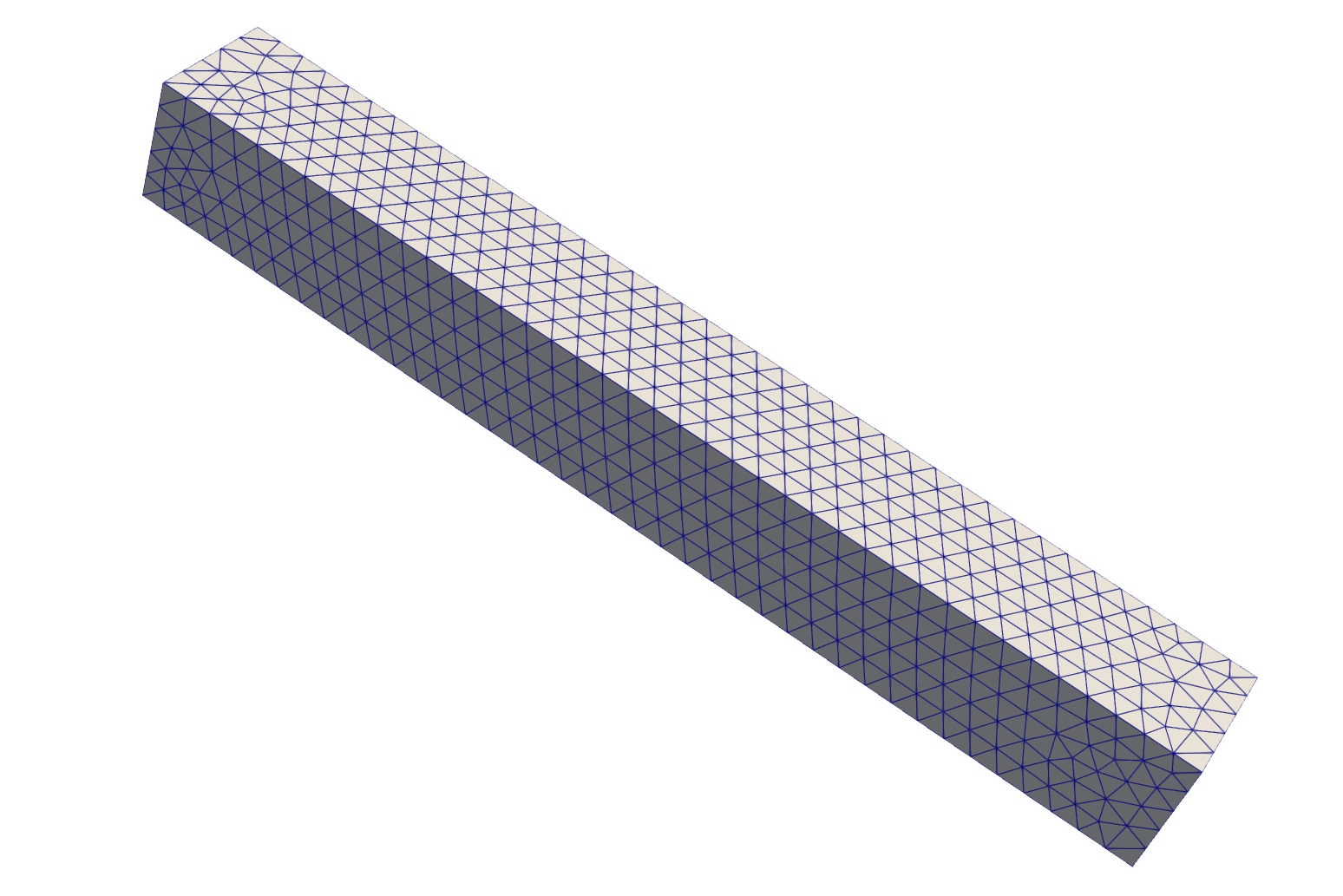}
        \caption{Initial surface}
    \end{subfigure}

    \vspace{-1pt}

    \begin{subfigure}[b]{0.3\textwidth}
        \centering
        \includegraphics[width=0.95\textwidth]{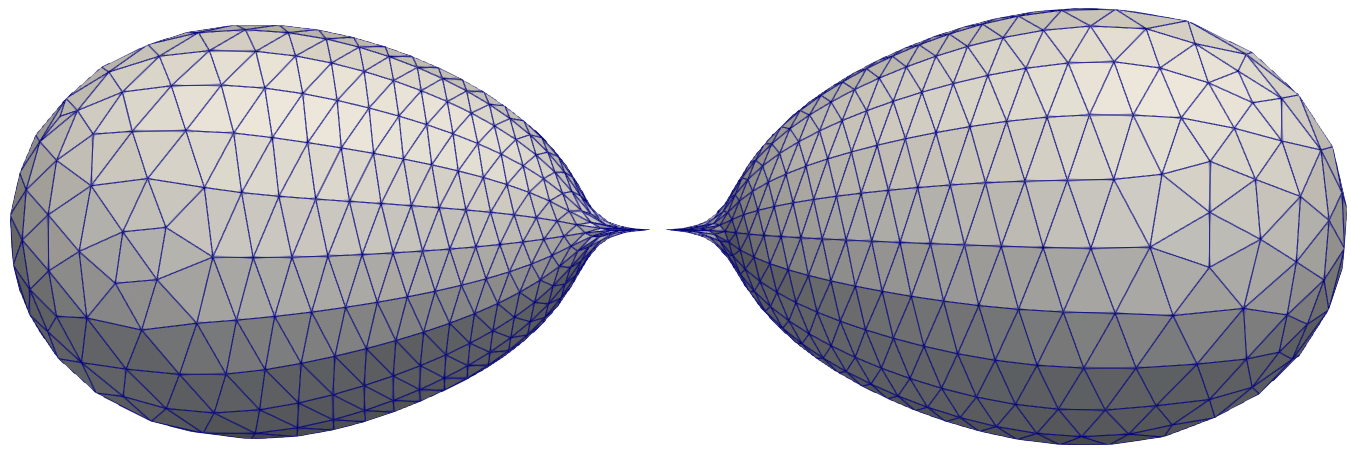}\vspace{5pt}
        \caption{BGN with $\tau = 10^{-3}$\newline\centering at $t = 0.366$}
        \label{fig:Box118-BGN}  
    \end{subfigure}
\begin{subfigure}[b]{0.3\textwidth}
    \centering
\includegraphics[width=0.95\textwidth]{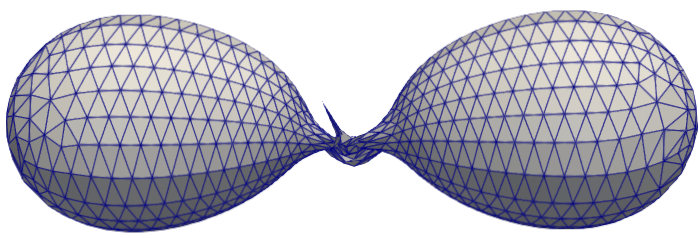}\vspace{5pt}
    \caption{MDR with $\tau = 10^{-3}$\newline\centering at $t = 0.366 $}
    \label{fig:Box118-MDR}
\end{subfigure}
 \begin{subfigure}[b]{0.3\textwidth}
        \centering
        \includegraphics[width=0.95\textwidth]{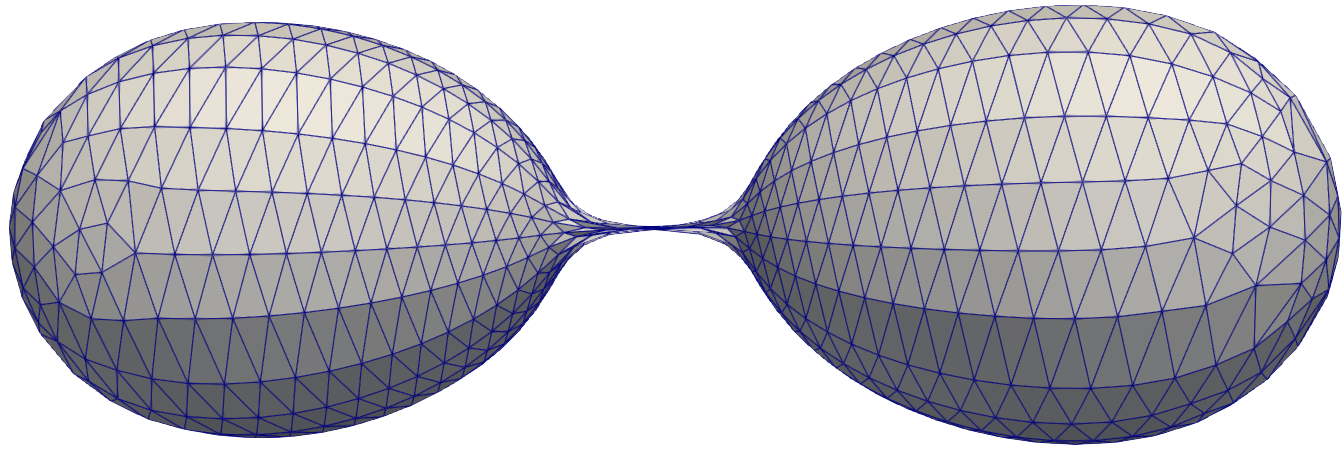}\vspace{5pt}
        \caption{Dual-MDR with\newline\centering $\tau = 10^{-3}$ at $t = 0.366$ }
        \label{fig:Box118-MDR-dual}
    \end{subfigure}

    \vspace{-1pt}
    
    \begin{subfigure}[b]{0.3\textwidth}
        \centering
        \includegraphics[width=0.95\textwidth]{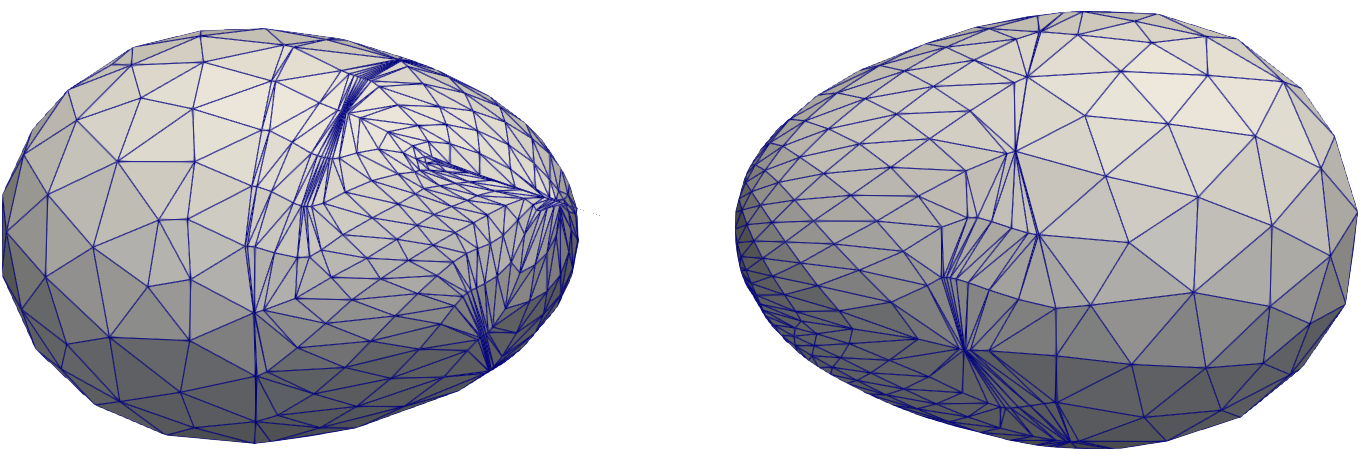}\vspace{5pt}
        \caption{BGN with $\tau = 10^{-4}$\newline\centering at $t = 0.366$}
         \label{fig:Box118_BGN2}
    \end{subfigure}
\begin{subfigure}[b]{0.3\textwidth}
    \centering
    \includegraphics[width=0.95\textwidth]{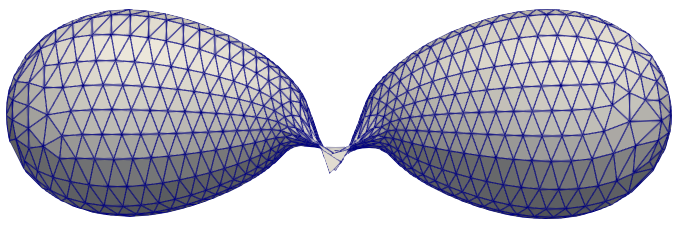}\vspace{5pt}
    \caption{MDR with $\tau = 10^{-4}$\newline\centering at $t = 0.3644$}
    \label{fig:Box118-MDR2}
\end{subfigure}
 \begin{subfigure}[b]{0.3\textwidth}
        \centering
        \includegraphics[width=0.95\textwidth]{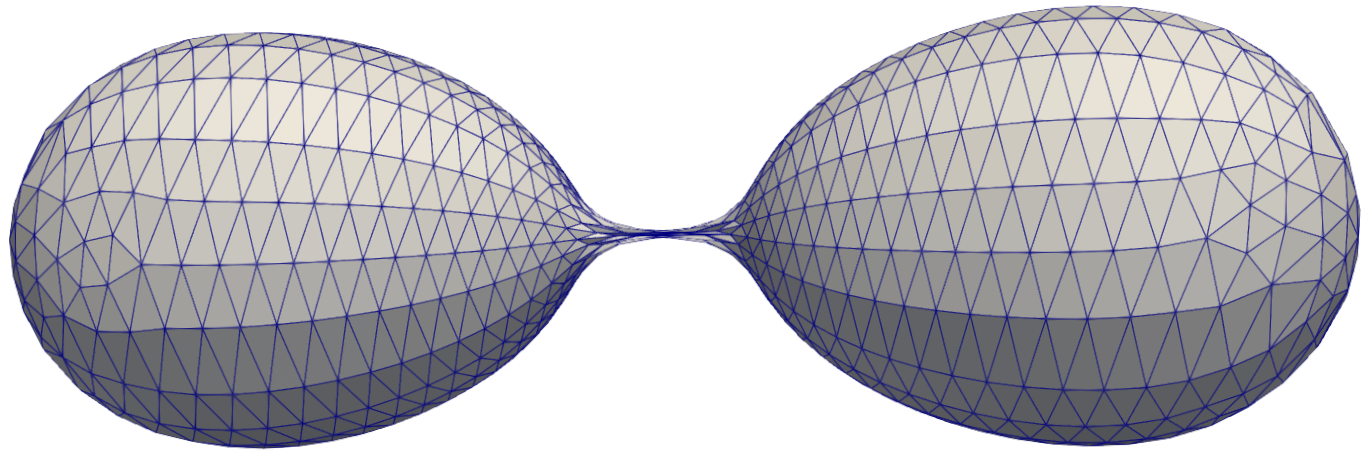}\vspace{5pt}
        \caption{Dual-MDR with\newline\centering $\tau = 10^{-4}$ at $t = 0.366$}
        \label{fig:Box118-MDR-dual2}
    \end{subfigure}

\caption{Comparison of surface evolution under surface diffusion in Example \ref{ex:surface diffusion-Box118}.}
\label{fig:Box118-Pic}
\end{figure}

In addition, Fig.~\ref{fig:Box118-Area} and Fig.~\ref{fig:Box118-dA} indicate that the dual-MDR scheme is unconditionally energy stable: the area decreases monotonically up to the pinch-off singularity. Finally, Fig.~\ref{fig:Box118MQ} further highlights the advantage of the dual-MDR scheme in preserving mesh quality up to the time of pinch-off singularity for $\tau=10^{-4}$.

\begin{figure}[htbp]
\centering
\captionsetup[subfigure]{skip=2pt} 

\begin{subfigure}[t]{0.333\textwidth}
  \centering
  \includegraphics[width=\linewidth]{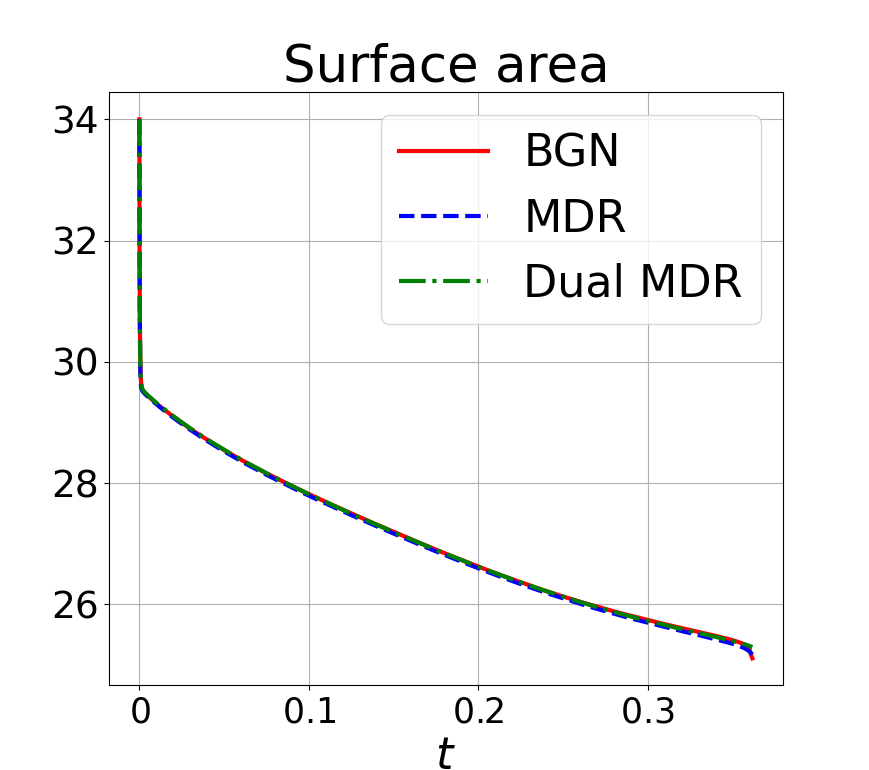}
  \caption{Surface area: $|\Gamma_h^{m}|$}
  \label{fig:Box118-Area}
\end{subfigure}\hfill
\begin{subfigure}[t]{0.333\textwidth}
  \centering
  \includegraphics[width=\linewidth]{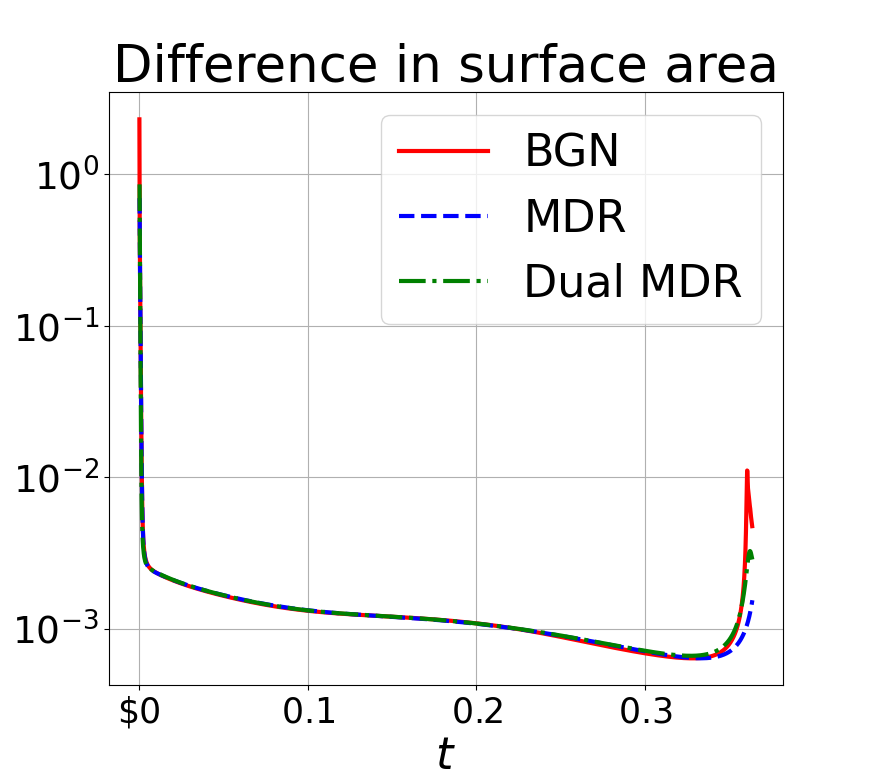}
  \caption{$|\Gamma_h^{m}| - |\Gamma_h^{m+1}|$}
  \label{fig:Box118-dA}
\end{subfigure}\hfill
\begin{subfigure}[t]{0.333\textwidth}
  \centering
  \includegraphics[width=\linewidth]{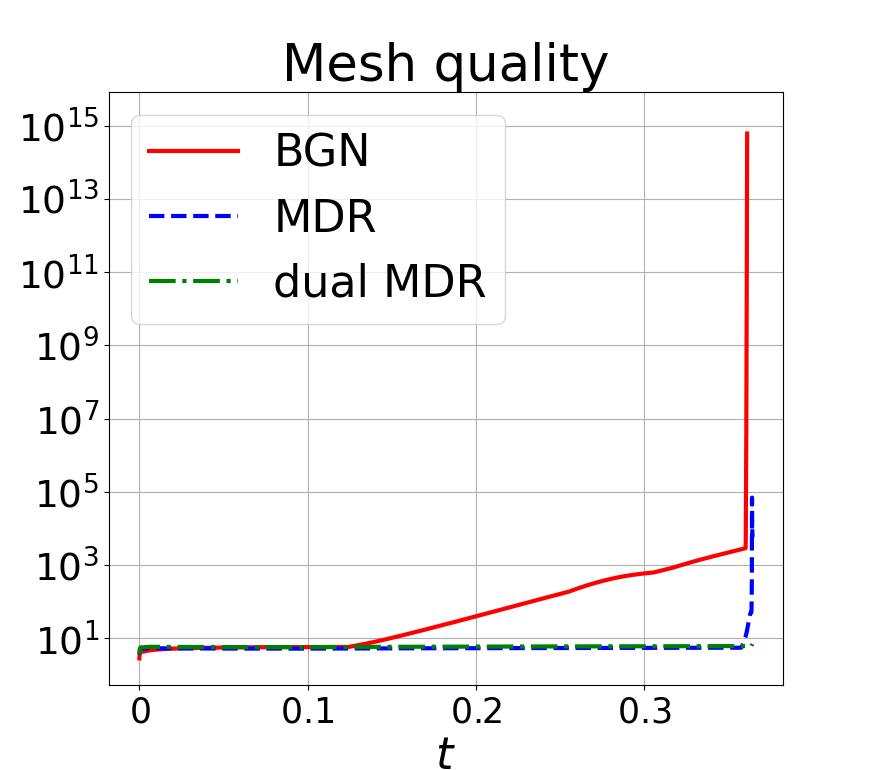}
  \caption{Mesh quality}
  \label{fig:Box118MQ}
\end{subfigure}
\caption{Comparison of the surface-area decay and mesh quality in Example~\ref{ex:surface diffusion-Box118}.}
\label{fig:Box118-AreaMQ}
\end{figure}

\end{example}

\section{Dual formulation for solid-state dewetting with a moving contact line}

The dual formulation approach extends naturally to open surfaces with a moving contact line, as encountered in solid-state dewetting on a substrate. 

%
%
Following the geometric framework in \cite{Bao2021,Bao2023}, we consider a thin-film surface $\Gamma(t)\subset\mathbb{R}^3\cap\{x=(x_1,x_2,x_3):x_3\ge 0\}$ intersecting the planar substrate $\{x=(x_1,x_2,x_3):x_3=0\}$ at the closed contact line:  
$$
\partial\Gamma(t)=\Gamma(t)\cap\{x\in\mathbb{R}^3:x_3=0\} .
$$
The evolution of $\Gamma=\Gamma(t)$ in solid-state dewetting is governed by the surface-diffusion law
\begin{align}\label{solid-state dewetting-normal-velocity}
v \cdot n = \Delta_{\Gamma} H
\quad \text{on } \Gamma
\end{align}
and supplemented by boundary conditions on $\partial\Gamma$: 
\begin{subequations}
\begin{align}
X_3(\cdot,t)\big|_{\partial\Gamma} &= 0
&&\text{(contact line stays on the substrate)}, \label{X_3} \\
\mu_{\scriptscriptstyle\partial} \cdot n_{\scriptscriptstyle\partial} &= \cos\theta &&\mbox{(Young's law for the contact angle)},\label{young} \\
\bigl(\mu_{\scriptscriptstyle\partial} \cdot \nabla_{\Gamma} H\bigr)\big|_{\partial\Gamma} &= 0 &&\mbox{(zero flux condition)},\label{non-flux}
\end{align}
\end{subequations}
where $\theta$ is a given parameter determined by the material of the thin film and the substrate, $\mu_{\scriptscriptstyle\partial}$ is the conormal vector, i.e., tangent to $\Gamma$ and normal to $\partial\Gamma$ (pointing downward), and $n_{\scriptscriptstyle\partial}$ is the unit vector normal to $\partial\Gamma$ and within the substrate plane.

Let \(w \in H^{1}(\Gamma)^{3}\) be an arbitrary test function. Testing \eqref{eq:id-H} with \(w\) and applying the surface
divergence theorem, we obtain
\begin{align*}
\int_{\Gamma} Hn \cdot w
&= \int_{\Gamma} \nabla_{\Gamma}\mathrm{id} \cdot \nabla_{\Gamma} w
   - \int_{\partial\Gamma} \mu_{\scriptscriptstyle\partial}\cdot w .
\end{align*}
To incorporate the contact-angle condition, we decompose \(\mu_{\scriptscriptstyle\partial}\) into its components in the
\(\{n_{\scriptscriptstyle\partial},e_{3}\}\)-directions:
\[
\mu_{\scriptscriptstyle\partial}
= (\mu_{\scriptscriptstyle\partial}\cdot n_{\scriptscriptstyle\partial})\,n_{\scriptscriptstyle\partial}
+ (\mu_{\scriptscriptstyle\partial}\cdot e_{3})\,e_{3},
\quad\mbox{where}\,\,\, e_{3}=(0,0,1).
\]
Using Young's law \eqref{young} and the identity \(\mu_{\scriptscriptstyle\partial}\cdot e_{3}=-\sin\theta\), we infer that
\begin{align*}
\int_{\partial\Gamma} \mu_{\scriptscriptstyle\partial}\cdot w
&= \int_{\partial\Gamma} \cos\theta\, (n_{\scriptscriptstyle\partial}\cdot w)
   - \int_{\partial\Gamma} \sin\theta\, (e_{3}\cdot w).
\end{align*}
Substituting this identity into the previous relation, we arrive at
\begin{align}\label{foundation}
\int_{\Gamma} Hn\cdot w
= \int_{\Gamma} \nabla_{\Gamma} {\rm id} \cdot \nabla_{\Gamma}w
  - \int_{\partial \Gamma} \cos\theta \,n_{\scriptscriptstyle\partial} \cdot w
  + \int_{\partial \Gamma} \sin\theta \,e_3 \cdot w,
\end{align}
which forms the basis of the variational formulation for the solid-state dewetting problem. 

Since the geometric evolution law \eqref{solid-state dewetting-normal-velocity} for solid-state dewetting prescribes only the normal component of the velocity, the tangential motion remains undetermined. In the next subsection we formulate the solid-state dewetting problem with the MDR tangential motion.

\subsection{MDR and dual-MDR formulations}

We note that \eqref{foundation} contains both information of mean curvature and contact angle. In particular, choosing $w=\psi n$ in \eqref{foundation} yields the following {\it weak formulation for determining the mean curvature $H$ and the conormal vector $\mu_\partial$ simultaneously}:
\begin{equation}\label{weakform:H-SD}
\int_{\Gamma} -H \psi + \int_{\Gamma} \nabla_{\Gamma} {\rm id} \cdot \nabla_{\Gamma} (\psi n)
= \int_{\partial \Gamma} (\cos \theta \,n_{\scriptscriptstyle\partial} - \sin \theta\, e_3) \cdot n \,\psi.
\end{equation}
Indeed, if $H$ is any scalar function satisfying \eqref{weakform:H-SD} then choosing $\psi\in C^\infty_0(\Gamma)$ and using integration by parts yields $H=-\Delta_{\Gamma}{\rm id}\cdot n$ (thus $H$ must be the mean curvature). Then, choosing $\psi\in C^\infty(\Gamma)$ and using integration by parts yields 
\begin{equation}
0 = \int_{\partial \Gamma} (\cos \theta \,n_{\scriptscriptstyle\partial} - \sin \theta\, e_3) \cdot n \,\psi \quad\forall\,\psi\in C^\infty(\Gamma).
\end{equation}
This shows that $\cos \theta \,n_{\scriptscriptstyle\partial} - \sin \theta\, e_3$ is orthogonal to $n$. Since $\cos \theta \,n_{\scriptscriptstyle\partial} - \sin \theta\, e_3$ is also orthogonal to the tangent vector of $\partial\Gamma$, it follows that $\cos \theta \,n_{\scriptscriptstyle\partial} - \sin \theta\, e_3=\mu_{\scriptscriptstyle\partial}$ (the sign is uniquely determined as $\mu_{\scriptscriptstyle\partial}$ is pointing downward). 

The argument above shows that \eqref{weakform:H-SD} determines $H$ as the mean curvature and simultaneously imposes Young's law \eqref{young}. Moreover, the constraint \eqref{X_3} requires the velocity to be in the space $
\mathbf{X}(\Gamma) := H^1(\Gamma) \times H^1(\Gamma) \times H^1_0(\Gamma) $. Therefore, combining \eqref{weakform:H-SD} with \eqref{solid-state dewetting-normal-velocity} and the zero flux condition \eqref{non-flux}, the continuous MDR formulation for solid-state dewetting can be written as finding \((v,H,\kappa) \in \mathbf{X}(\Gamma) \times H^1(\Gamma) \times H^1(\Gamma)\) such that
\begin{subequations}\label{eq:solid-state dewetting-cont}
\begin{align}
\label{eq:solid-state dewetting-H-cont}
\int_{\Gamma} -H \psi + \int_{\Gamma} \nabla_{\Gamma} {\rm id} \cdot \nabla_{\Gamma} (n\psi) &= \int_{\partial \Gamma} (\cos \theta \,n_{\scriptscriptstyle\partial} - \sin \theta\, e_3)\cdot n \,\psi, \\
\label{eq:solid-state dewetting-v-cont}
\int_{\Gamma} (v\cdot n) \phi + \int_{\Gamma} \nabla_{\Gamma} H \cdot \nabla_{\Gamma} \phi &= 0, \\
\label{eq:solid-state dewetting-kappa-cont}
  a_{\Gamma}(v,\eta) &= \int_{\Gamma} \kappa n \cdot \eta,
\end{align}
\end{subequations}
for all \((\psi,\phi,\eta) \in H^1(\Gamma) \times H^1(\Gamma) \times \mathbf{X}(\Gamma)\), where $a_{\Gamma}(\cdot,\cdot)$ is a bilinear form for defining a tangential motion, as shown in \eqref{sym-MDR}.

The MDR formulation \eqref{eq:solid-state dewetting-cont} can determine the open-surface evolution in solid-state dewetting with the MDR tangential motion for mesh-quality control. However, its direct discretization does not guarantee the desired energy stability. This can again be achieved by considering the corresponding dual-MDR formulation. 

In view of the dual-MDR formulation of closed-surface evolution in surface diffusion, we can directly write down the dual-MDR formulation for the solid-state dewetting problem as follows (in the weak form): 
%
Find \((v,H,\lambda,\kappa) \in \mathbf{X}(\Gamma) \times H^1(\Gamma) \times \mathbf{X}(\Gamma) \times H^1(\Gamma)\) such that
\begin{subequations}\label{eq:solid-state dewetting}
\begin{align}
\label{eq:solid-state dewetting-H}
&\int_{\Gamma} -H n \cdot w + \int_{\Gamma} \nabla_{\Gamma} {\rm id} \cdot \nabla_{\Gamma} w - \int_{\partial \Gamma} \cos\theta \,n_{\scriptscriptstyle\partial} \cdot w
= a_{\Gamma}(\lambda, w) ,\\
\label{eq:solid-state dewetting-v}
&\int_{\Gamma} (v\cdot n)\,\phi + \int_{\Gamma} \nabla_{\Gamma} H \cdot \nabla_{\Gamma} \phi
= 0,\\
\label{eq:solid-state dewetting-kappa}
&  a_{\Gamma}(v,\eta)
= \int_{\Gamma} \kappa n \cdot \eta,\\
\label{eq:solid-state dewetting-lambda}
&\int_{\Gamma} \lambda \cdot n \,\varphi
= 0,
\end{align}
\end{subequations}
for all \((w,\phi,\eta,\varphi) \in \mathbf{X}(\Gamma) \times H^1(\Gamma) \times \mathbf{X}(\Gamma) \times H^1(\Gamma)\). Here, the dual multiplier \(\lambda\equiv 0\) is introduced exactly as in the closed-surface case, with its exact solution being zero, and \eqref{eq:solid-state dewetting-H} is obtained from \eqref{foundation} by using the property \(w\cdot e_3 = 0\) on \(\partial\Gamma\) for \(w \in \mathbf{X}(\Gamma)\). The equivalence between \eqref{eq:solid-state dewetting} and \eqref{eq:solid-state dewetting-cont} can be shown similarly as in the closed-surface case.

\subsection{Linearly implicit dual-MDR scheme for solid-state dewetting}

Let $\mathbf{X}_h(\Gamma_h^{m}) := S_h(\Gamma_h^{m}) \times S_h(\Gamma_h^{m}) \times \mathring{S}_h(\Gamma_h^{m})$, with \(\mathring{S}_h(\Gamma_h^{m}) \) being the space of finite element functions vanishing on \(\partial\Gamma_h^m\). Then the  discretization of the dual-MDR formulation \eqref{eq:solid-state dewetting} leads to the dual-MDR scheme for the solid-state dewetting problem:  
%
Find \(( v_h^{m+1},  H_h^{m+1}, \lambda_h^{m+1}, \kappa_h^{m+1}) \in \mathbf{X}_h(\Gamma_h^{m}) \times S_h(\Gamma_h^{m}) \times \mathbf{X}_h(\Gamma_h^{m}) \times S_h(\Gamma_h^{m})\) such that
\begin{subequations}\label{eq:num-solid-state dewetting}
\begin{align}
\label{eq:num-solid-state dewetting-H}
& \int_{\Gamma_h^{m}}^{(h)}  - H_h^{m+1} n_h^{m} \cdot w_h + \int_{\Gamma_h^{m}} \nabla_{\Gamma_h^{m}}({\rm id} + \tau  v_h^{m+1}) \cdot \nabla_{\Gamma_h^{m}} w_h  - \int_{\partial \Gamma_h^{m}}\cos\theta \,n_{\scriptscriptstyle\partial}^{m+\frac{1}{2}} \cdot w_h\notag\\
&= a_{\Gamma_h^{m}}(\lambda_h^{m+1}, w_h), \\
\label{eq:num-solid-state dewetting-v}
& \int_{\Gamma_h^{m}}^{(h)} ( v_h^{m+1}\cdot n_h^{m}) \phi_h + \int_{\Gamma_h^{m}}\nabla_{\Gamma_h^{m}}  H_h^{m+1} \cdot \nabla_{\Gamma_h^{m}} \phi_h
=0, \\
\label{eq:num-solid-state dewetting-kappa}
& a_{\Gamma_h^{m}}(v_h^{m+1},\eta_h) 
= \int_{\Gamma_h^{m}}^{(h)}  \kappa_h^{m+1} n_h^{m} \cdot \eta_h, \\
 \label{eq:num-solid-state dewetting-lambda}
&\int_{\Gamma_h^m}^{(h)}  \lambda_h^{m+1} \cdot n_h^{m} \varphi_h = 0,
\end{align}
\end{subequations}
for all \((w_h,\phi_h,\eta_h,\varphi_h) \in \mathbf{X}_h(\Gamma_h^{m}) \times S_h(\Gamma_h^{m}) \times \mathbf{X}_h(\Gamma_h^{m}) \times S_h(\Gamma_h^{m})\).
%
Here, we define the discrete conormal vector \(n_{\scriptscriptstyle\partial}^{m+\frac12}\) by midpoint averaging of the boundary tangents as in \cite[Eq.~(3.7)]{Bao2023} and \cite{jiang2021perimeter}: if \(s\) denotes the arc-length parameter along the discrete boundary curve, then
\begin{align}\label{n-partial-def}
n_{\scriptscriptstyle\partial}^{m+\frac12}
:= \frac{1}{2}\bigl(\partial_s {\rm id} + \partial_s X_h^{m+1}\bigr)\times e_3.
\end{align}

The following theorem shows that the scheme \eqref{eq:num-solid-state dewetting} is unconditionally energy stable, and hence reflects the underlying physical dissipation mechanism.

\begin{theorem}\label{thm:area-decreasing-mass-sd-3D}
The numerical solution determined by the dual-MDR scheme \eqref{eq:num-solid-state dewetting} has monotonically non-increasing energy, i.e., 
\begin{align}\label{thm:area-decreasing-mass-sd-1-3D}
|\Gamma_h^{m+1}| - \cos \theta \,|S_1^{m+1}| \le |\Gamma_h^m| - \cos \theta \,|S_1^m| ,
\end{align}
where \(S_1^m\) denotes the planar domain enclosed by the contact line \(\partial\Gamma_h^m\) on the substrate.
\end{theorem}

\begin{proof}
Choosing $(w_h,\phi_h,\eta_h,\varphi_h) =  (v_h^{m+1},H_h^{m+1},\lambda_h^{m+1},\kappa_h^{m+1})$ in \eqref{eq:num-solid-state dewetting}, and summing the resulting identities, we obtain
\begin{align*}
    &\int_{\Gamma_h^{m}} \nabla_{\Gamma_h^{m}} ({\rm id}+\tau v_h^{m+1})
    \cdot \nabla_{\Gamma_h^{m}}  v_h^{m+1}
    - \cos \theta
    \int_{\partial \Gamma_h^{m}} n_{\scriptscriptstyle\partial}^{m+\frac{1}{2}}
    \cdot v_h^{m+1} \\
    &=- \|\nabla_{\Gamma_h^{m}}  H_h^{m+1} \|_{L^2(\Gamma_h^m)}^2 
    \leq 0 .
\end{align*}
Rewriting ${\rm id}+\tau v_h^{m+1}$ as $X_h^{m+1}$ and $v_h^{m+1}$ as $(X_h^{m+1}-{\rm id})/\tau$, we have 
\begin{align*}
    &\int_{\Gamma_h^{m}} \nabla_{\Gamma_h^{m}} X_h^{m+1}
    \cdot \nabla_{\Gamma_h^{m}}  v_h^{m+1}
    - \frac{1}{\tau}\cos \theta
    \int_{\partial \Gamma_h^{m}} n_{\scriptscriptstyle\partial}^{m+\frac{1}{2}}
    \cdot (X_h^{m+1}-{\rm id}) \\
    &=- \|\nabla_{\Gamma_h^{m}}  H_h^{m+1} \|_{L^2(\Gamma_h^m)}^2 
    \leq 0 .
\end{align*}
Then, using \eqref{area-ineq} and the following identity (see \cite[Lemma~3.1]{Bao2023}):
\begin{align}\label{eq:boundary-area-solid-state dewetting}
    \int_{\partial \Gamma_h^{m}} n_{\scriptscriptstyle\partial}^{m+\frac{1}{2}}
    \cdot (X_h^{m+1} - {\rm id})
    = |S_1^{m+1}| - |S_1^{m}| ,
\end{align}
we obtain \eqref{thm:area-decreasing-mass-sd-1-3D}. 
\hfill\end{proof}

\subsection{Numerical experiments}\label{sec:numerical_1}
In this section, we present numerical experiments to illustrate the performance of the proposed dual-MDR schemes for open-surface evolution with a moving contact line on a substrate in the solid-state dewetting setting. 

\begin{example}\label{Example7}\upshape
We present numerical simulations of solid-state dewetting problem for an open surface with a moving contact line constrained to the plane $z=0$, where the initial surface is a rectangular box of size $1 \times 6 \times 1$, centered at $(0,0,0)$, with prescribed contact angle $120^\circ$, and the evolution is computed using the classical BGN scheme, the MDR scheme \eqref{eq:solid-state dewetting-cont}, and the dual-MDR scheme \eqref{eq:num-solid-state dewetting}. 

For mesh size $h=0.2$ and time-step size $\tau=10^{-2}$, both the MDR and dual-MDR schemes remain stable and maintain satisfactory mesh quality throughout the evolution; see Figures~\ref{Example7c_60-scalar} and~\ref{Example7c_60}; by contrast, the BGN scheme exhibits noticeable rotational artifacts near the substrate, which degrade the mesh quality in the vicinity of the contact line; see Figure~\ref{Example7b_60}. 

When the time-step size is reduced to $\tau=10^{-3}$, the BGN scheme becomes unstable and breaks down at time $T=0.277$; see Figure~\ref{Example7d_60}; in contrast, both the dual-MDR scheme and the MDR scheme remain stable up to the final time $T=2$ while preserving good mesh quality; see Figures~\ref{Example7e_60} and~\ref{Example7e_6_scalar}. For improved visualization, the surface evolution plots in this example are shown with appropriate scaling.

\begin{figure}[htbp]
    \centering
    \begin{subfigure}[b]{0.5\textwidth}
        \centering
        \includegraphics[width=0.85\textwidth]{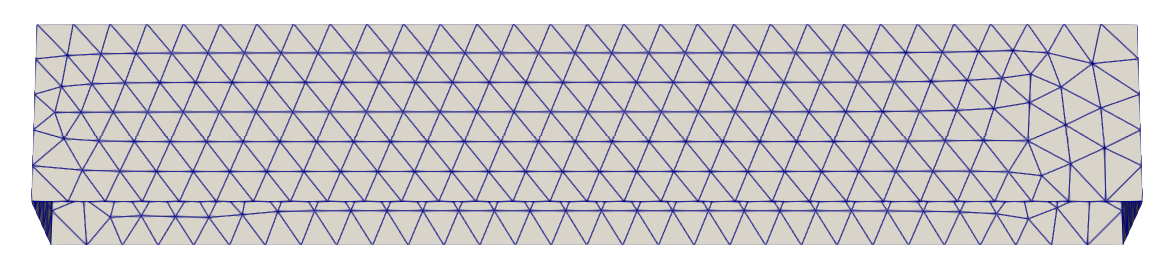}
        \caption{Initial surface}
        \label{Example7a_60}
    \end{subfigure}

    \vspace{10pt}

    \begin{subfigure}[b]{0.3\textwidth}
        \centering
        \includegraphics[width=0.6\textwidth]{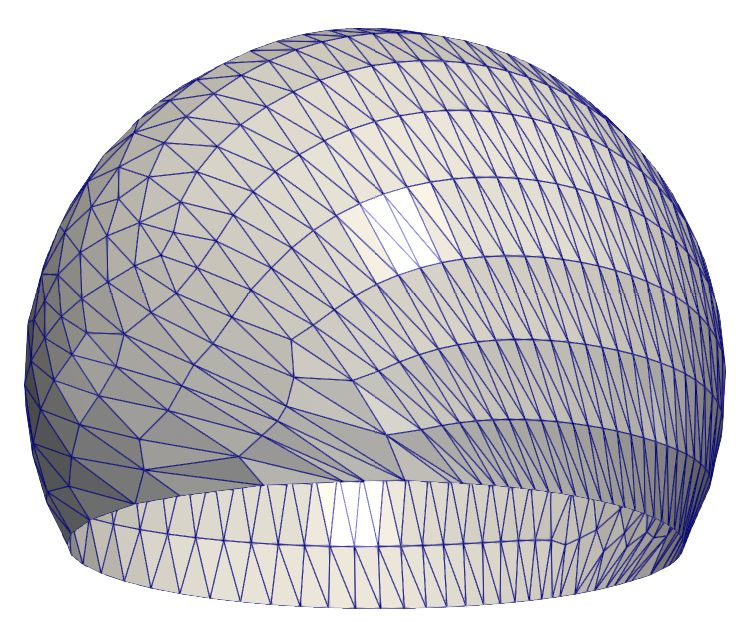}
        \caption{BGN at $T=2$ \\ with $\tau=10^{-2}$}
        \label{Example7b_60}
    \end{subfigure}
    \begin{subfigure}[b]{0.3\textwidth}
        \centering
        \includegraphics[width=0.6\textwidth]{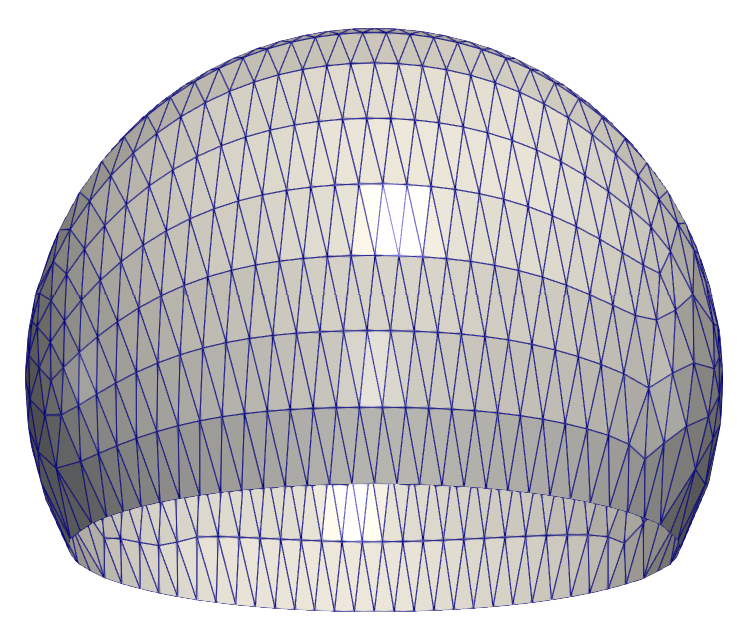}
        \caption{MDR at $T=2$ \\ with $\tau=10^{-2}$}
        \label{Example7c_60-scalar}
    \end{subfigure}
 \begin{subfigure}[b]{0.3\textwidth}
        \centering
        \includegraphics[width=0.6\textwidth]{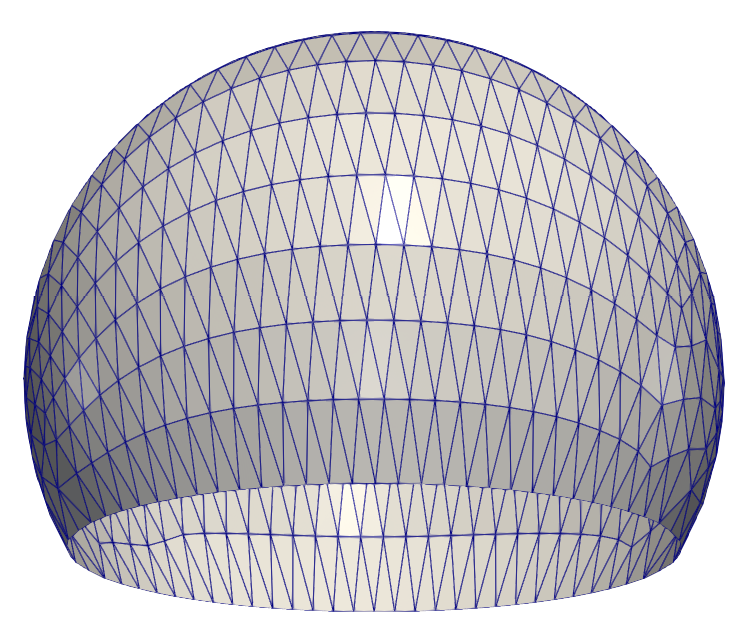}
        \caption{Dual-MDR at $T=2$ \\ with $\tau=10^{-2}$}
        \label{Example7c_60}
    \end{subfigure}
    \vspace{10pt}

    \begin{subfigure}[b]{0.3\textwidth}
        \centering
        \includegraphics[width=0.8\textwidth]{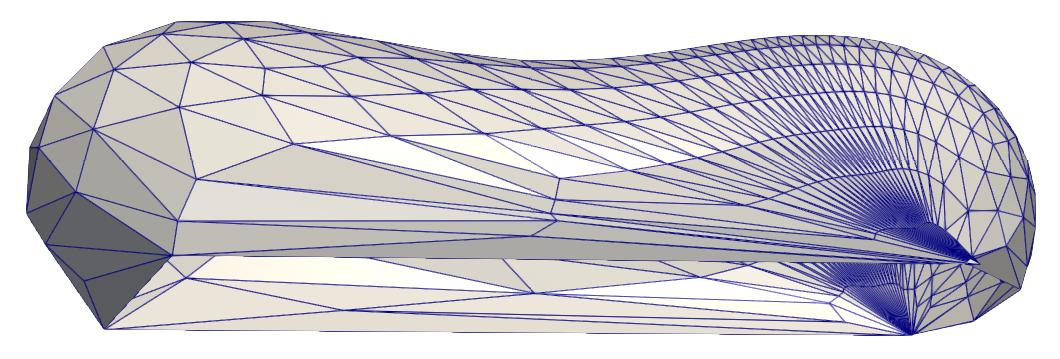}
        \caption{BGN  at $T=0.277$ \\ with $\tau=10^{-3}$}
        \label{Example7d_60}
    \end{subfigure}
    \begin{subfigure}[b]{0.3\textwidth}
        \centering
        \includegraphics[width=0.6\textwidth]{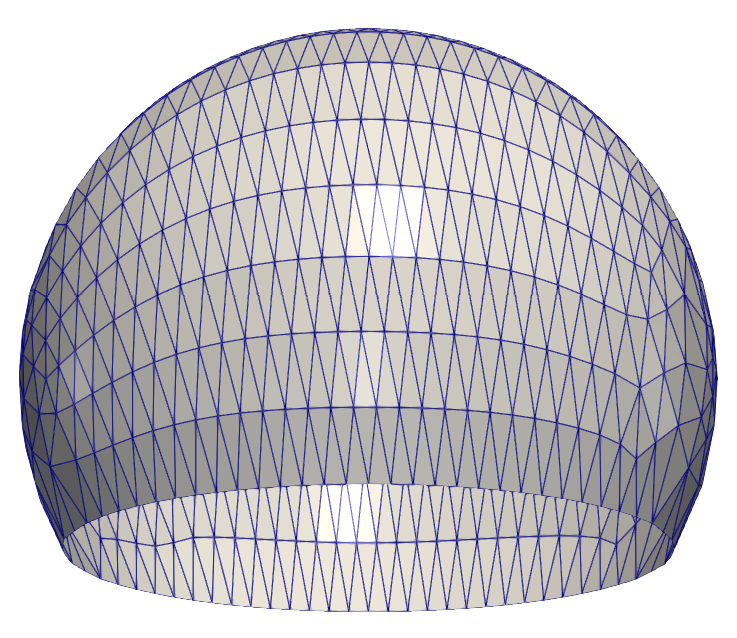}
        \caption{MDR at $T=2$ \\ with $\tau=10^{-3}$}
        \label{Example7e_6_scalar}
    \end{subfigure}
    \begin{subfigure}[b]{0.3\textwidth}
        \centering
        \includegraphics[width=0.6\textwidth]{Figures/120_Cube_1_1_6_13versionparaview/MDR_normal_tau1e-3_h0.2_MDR_normal_dual.png}
        \caption{Dual-MDR at $T=2$ \\ with $\tau=10^{-3}$ }
        \label{Example7e_60}
    \end{subfigure}
    \caption{Solid-state dewetting with a $120^\circ$ contact angle in Example~\ref{Example7}.}
    \label{eg:Example7_60}
\end{figure}

Figures~\ref{surface_area_example7} and~\ref{Example7e_120_area_difference} present a comparison of the surface-area decay for the BGN, MDR, and dual-MDR schemes with $\tau=10^{-3}$; moreover, as shown in Figure~\ref{surface_area_example7}, the dual-MDR scheme yields a smaller discrete surface area at each time step than both the BGN and MDR schemes. Figure~\ref{mesh_quality_example7} compares the mesh-quality metrics for the same three schemes with $\tau=10^{-3}$, and the results demonstrate the effectiveness of the proposed dual-MDR scheme in preserving mesh quality throughout the evolution, particularly in the presence of initial incompatibility, since the contact angle of the initial surface is $90^\circ$, whereas the equilibrium contact angle is $120^\circ$.

\begin{figure}[htbp]
\centering
\captionsetup[subfigure]{skip=2pt}

\begin{subfigure}[t]{0.333\textwidth}
  \centering
  \includegraphics[width=\linewidth]{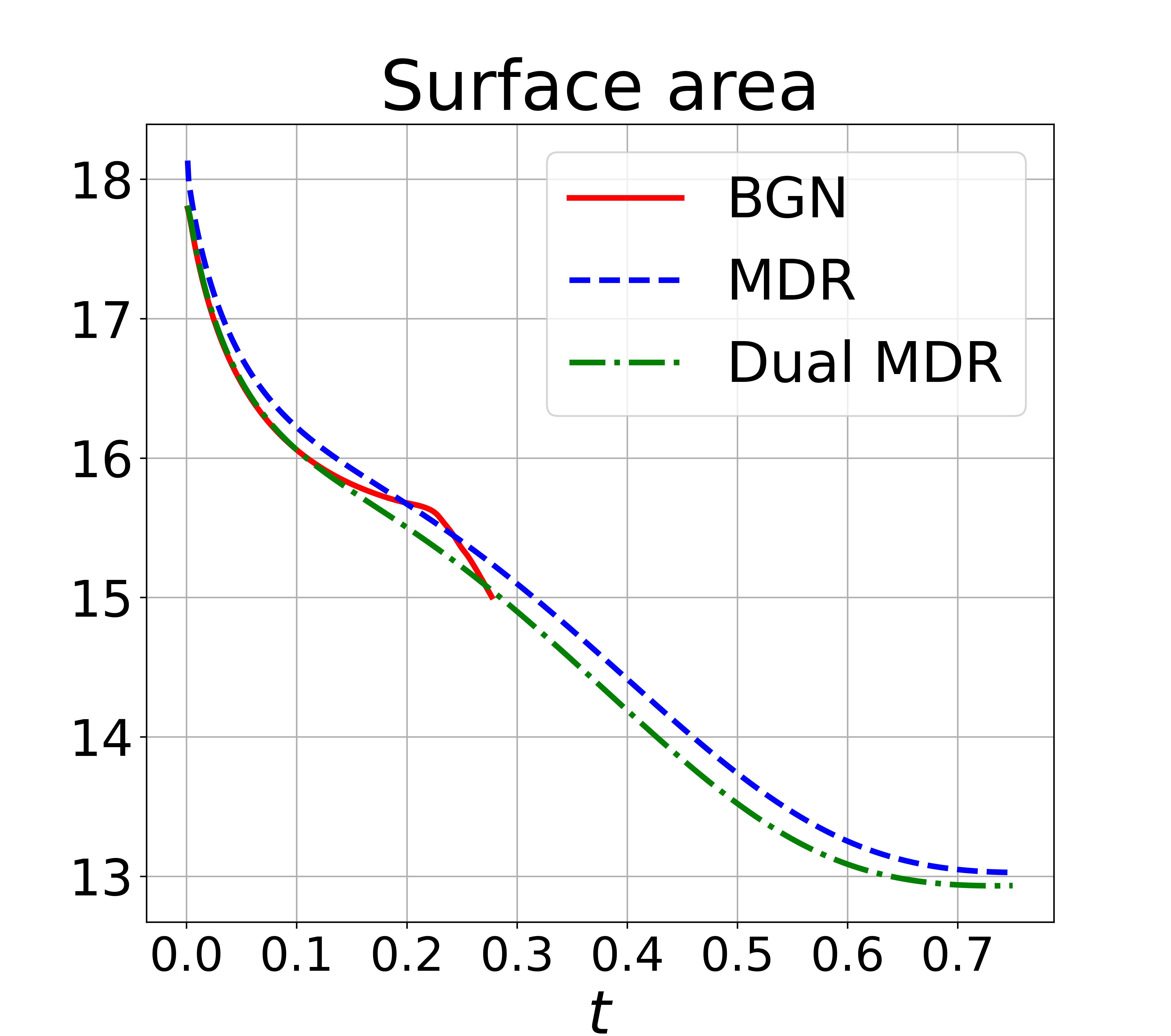}
  \caption{Surface area: $|\Gamma_h^m|$}
\label{surface_area_example7}
\end{subfigure}\hfill
\begin{subfigure}[t]{0.333\textwidth}
  \centering
  \includegraphics[width=\linewidth]{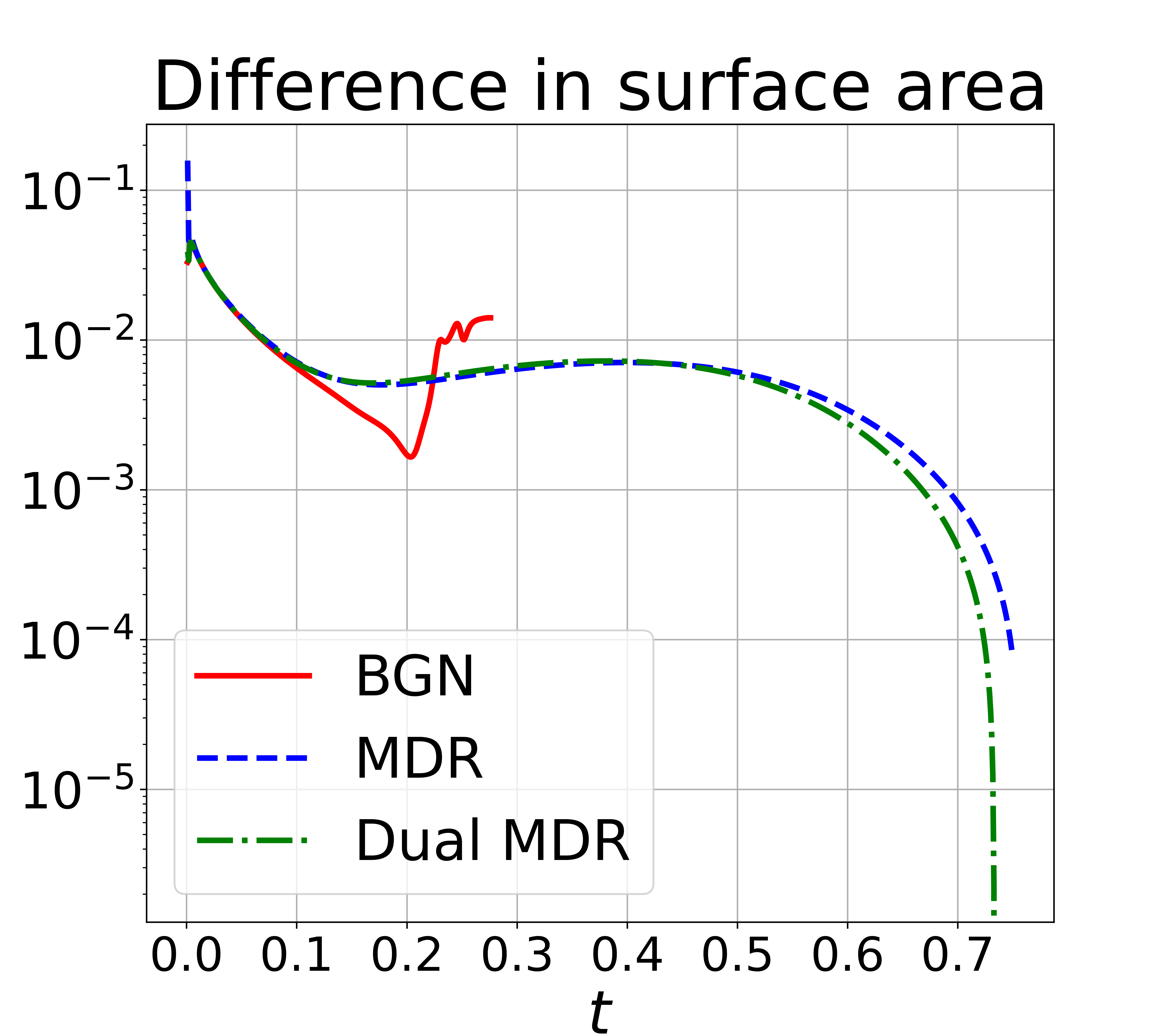}
  \caption{$|\Gamma_h^m| - |\Gamma_h^{m+1}|$}
  \label{Example7e_120_area_difference}
\end{subfigure}\hfill
\begin{subfigure}[t]{0.333\textwidth}
  \centering
  \includegraphics[width=\linewidth]{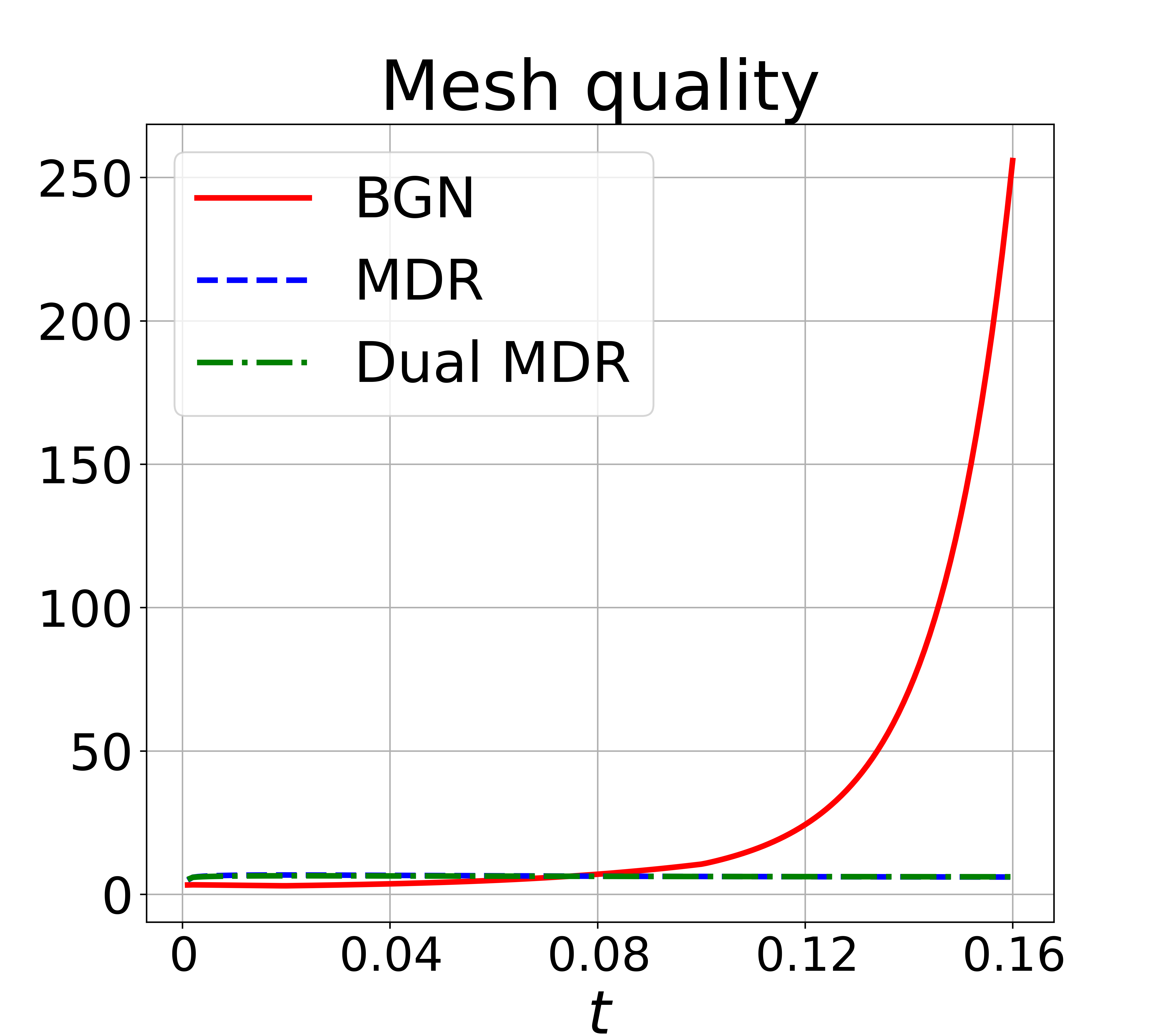}
  \caption{Mesh quality}
\label{mesh_quality_example7}
\end{subfigure}

\caption{Comparison of the surface-area decay and mesh quality in Example~\ref{Example7}.}
\label{eg:Example7_120}
\end{figure}


\end{example}

\begin{example}
\label{Example1-1-16}\upshape
We present numerical simulations of the solid-state dewetting problem for an open surface with a moving contact line constrained to the plane $z=0$, with prescribed equilibrium contact angle $90^\circ$, where the initial surface is a rectangular box of size $1 \times 1 \times 16$ centered at the origin, and all computations are performed on a quasi-uniform mesh with mesh size $h=0.2$. The evolution is computed using the BGN scheme, the MDR scheme \eqref{eq:solid-state dewetting-cont}, and the generalized dual-MDR scheme with the symmetric-gradient version of MDR tangential motion defined in \eqref{sym-MDR}.

For $\tau=10^{-2}$, all three schemes remain stable up to time $T=3.3$. The BGN scheme captures the pinch-off clearly; see Figure~\ref{Example1-1-16b}; however, the computation breaks down immediately afterward and cannot be continued to simulate the long-time evolution or compute the steady state. By contrast, the dual-MDR scheme preserves good mesh quality at $T=3.3$ and continues beyond pinch-off, thereby enabling long-time computation toward equilibrium; see Figures~\ref{Example1-1-16i} and~\ref{Example1-1-16j}. For the MDR scheme, the pinch-off is not completed numerically: at $T=3.33$, the two bulk components remain connected, and severe mesh distortion develops in the neck region, leading to failure at the next time step; see Figure~\ref{Example1-1-16c}.

For $\tau=10^{-3}$, the differences become more pronounced. In the BGN scheme, mesh points cluster near the pinch-off location, producing mesh distortion and causing the computation to break down shortly after the pinch-off time $T=3.243$; see Figure~\ref{Example1-1-16e}. The MDR scheme likewise fails to resolve the pinch-off accurately: severe mesh distortion develops as the solution approaches the pinch-off time $T=3.279$, eventually terminating the simulation; see Figure~\ref{Example1-1-16f}. By contrast, the dual-MDR scheme remains stable up to the final time $T=4$ while maintaining good mesh quality throughout the evolution; see Figures~\ref{Example1-1-16g} and~\ref{Example1-1-16m}. For improved visualization, the surface evolution plots in this example are shown with appropriate scaling.

\begin{figure}[htbp]
    \centering
    \captionsetup[subfigure]{skip=2pt}
    \begin{subfigure}[t]{0.6\textwidth}
        \centering
        \includegraphics[width=1\textwidth]{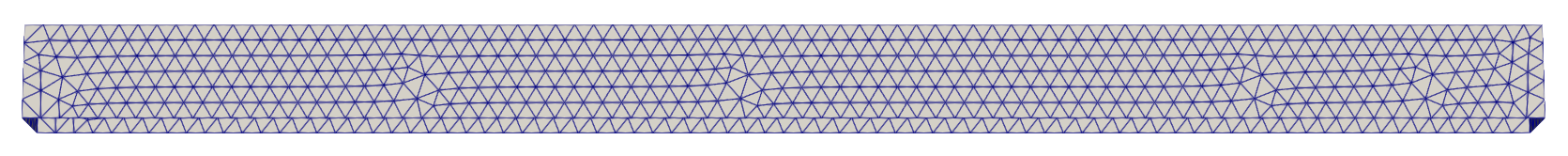}
        \caption{Initial surface}
        \label{Example1-1-16a}
    \end{subfigure}
    \vspace{10pt}

    \begin{subfigure}[t]{0.32\textwidth}
        \centering
        \includegraphics[width=1\textwidth]{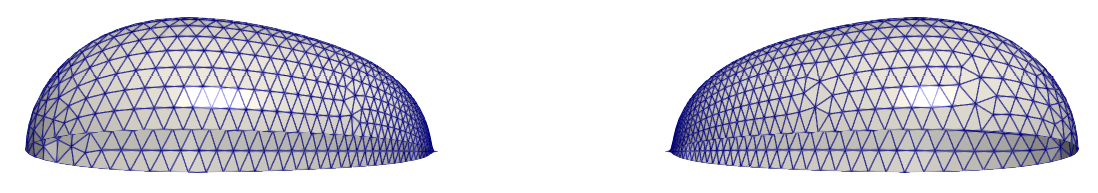}
        \caption{BGN at $T=3.34$ \\with $\tau=10^{-2}$ }
        \label{Example1-1-16b}
    \end{subfigure}
    \hfill
    \begin{subfigure}[t]{0.32\textwidth}
        \centering
        \includegraphics[width=1\textwidth]{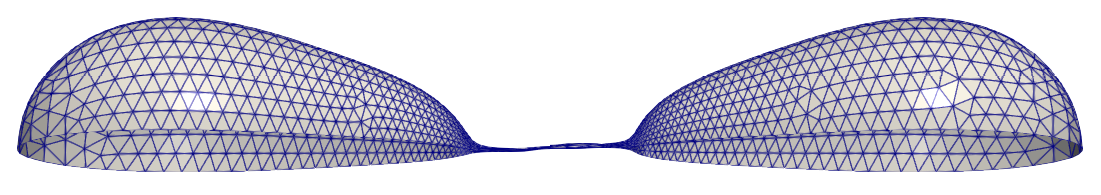}
        \caption{MDR at $T=3.33$ \\with $\tau=10^{-2}$ }
        \label{Example1-1-16c}
    \end{subfigure}
    \hfill
    \begin{subfigure}[t]{0.32\textwidth}
        \centering
        \includegraphics[width=1\textwidth]{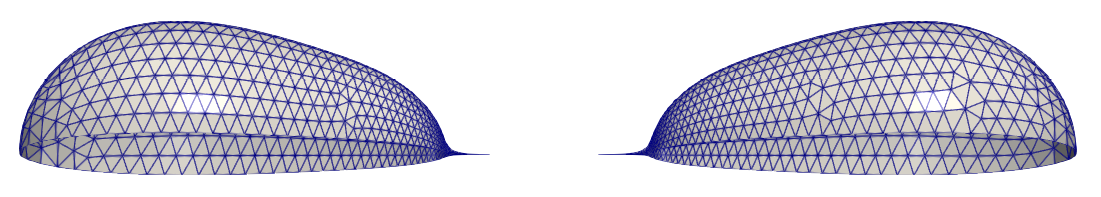}
        \caption{Dual-MDR at $T=3.3$ \\ with $\tau=10^{-2}$}
        \label{Example1-1-16d}
    \end{subfigure}

\vspace{10pt}

\begin{subfigure}[t]{0.32\textwidth}
        \centering
        \includegraphics[width=1\textwidth]{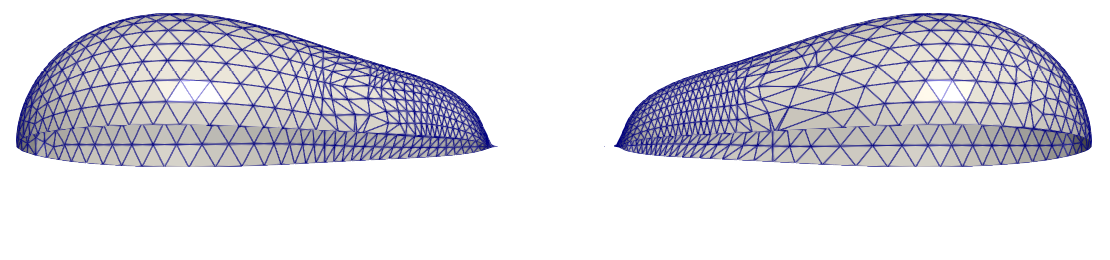}
        \caption{BGN at $T=3.243$ \\with $\tau=10^{-3}$ }
        \label{Example1-1-16e}
    \end{subfigure}
    \hfill
    \begin{subfigure}[t]{0.32\textwidth}
        \centering
        \includegraphics[width=1\textwidth]{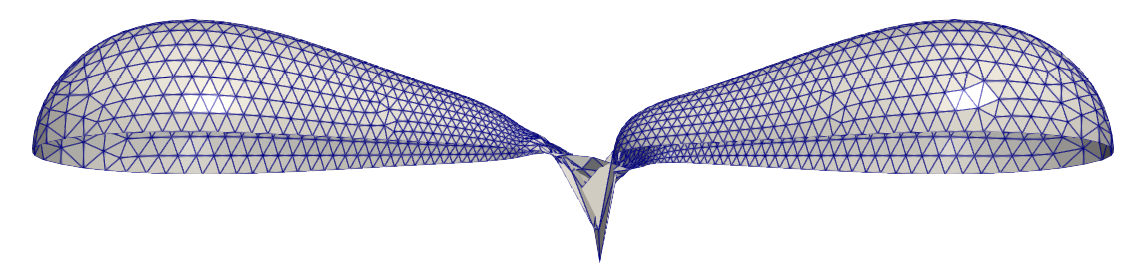}
        \caption{MDR at $T=3.279$ \\with $\tau=10^{-3}$ }
        \label{Example1-1-16f}
    \end{subfigure}
    \hfill
    \begin{subfigure}[t]{0.32\textwidth}
        \centering
        \includegraphics[width=1\textwidth]{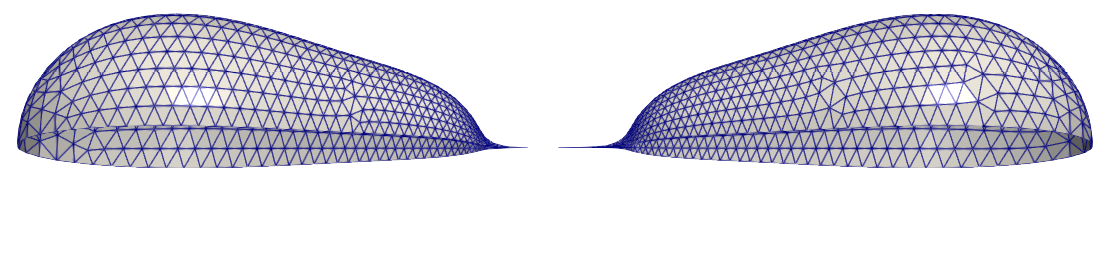}
        \caption{Dual-MDR at $T=3.25$ \\ with $\tau=10^{-3}$}
        \label{Example1-1-16g}
    \end{subfigure}

    \caption{Solid-state dewetting with a $90^\circ$ contact angle in Example~\ref{Example1-1-16}.}
    \label{eg:Example1-1-16}
\end{figure}

\begin{figure}[htbp]
    \centering
    \begin{subfigure}[b]{0.3\textwidth}
        \centering
        \includegraphics[width=1\textwidth]{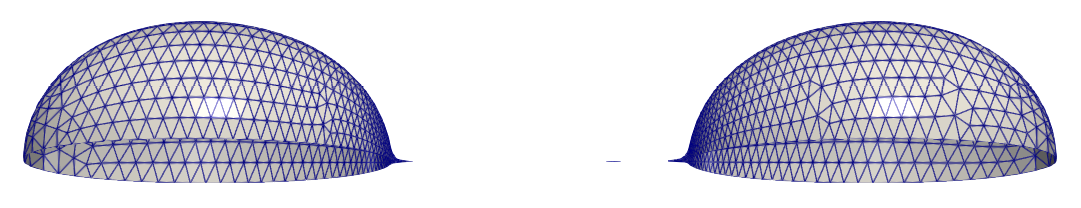}
        \caption{Dual-MDR at $T=3.5$ \\with $\tau=10^{-2}$ }
        \label{Example1-1-16i}
    \end{subfigure}
    \hspace{5pt}
    \begin{subfigure}[b]{0.3\textwidth}
        \centering
        \includegraphics[width=1\textwidth]{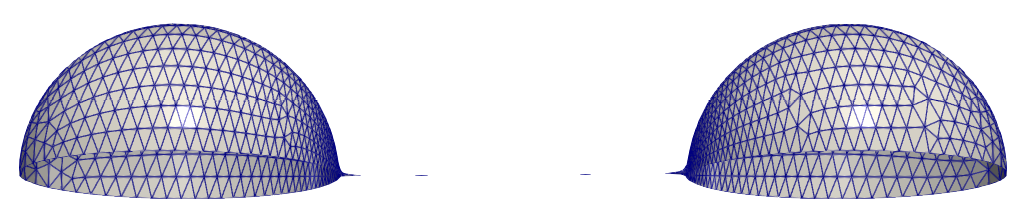}
        \caption{Dual-MDR at $T=4$ \\ with $\tau=10^{-2}$}
        \label{Example1-1-16j}
    \end{subfigure}
     \hspace{5pt}
    \begin{subfigure}[b]{0.3\textwidth}
        \centering
        \includegraphics[width=1\textwidth]{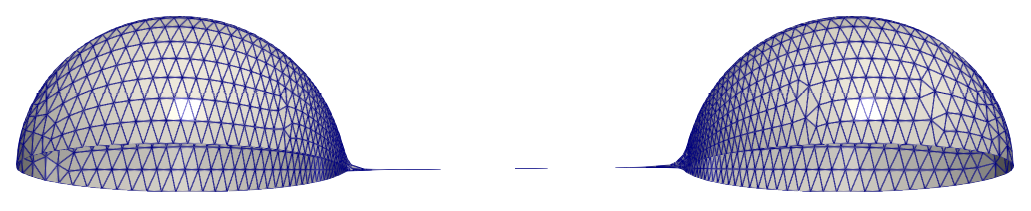}
        \caption{Dual-MDR at $T=4$ \\ with $\tau=10^{-3}$}
        \label{Example1-1-16m}
    \end{subfigure}
    \caption{Long-time behavior of the dual-MDR scheme in Example~\ref{Example1-1-16}. }
\label{long_time_example9}
\end{figure}

For $\tau = 10^{-3}$, Figures~\ref{Example9d_90_area} and~\ref{Example9e_90_area_difference} show the evolution of the surface area for the three methods, while Figure~\ref{mesh_quality_example9} compares the corresponding mesh-quality metrics. The results clearly demonstrate the superior performance of the proposed dual-MDR scheme in preserving mesh quality throughout the evolution.

\begin{figure}[htbp]
\centering
\captionsetup[subfigure]{skip=2pt}

\begin{subfigure}[t]{0.333\textwidth}
  \centering
  \includegraphics[width=\linewidth]{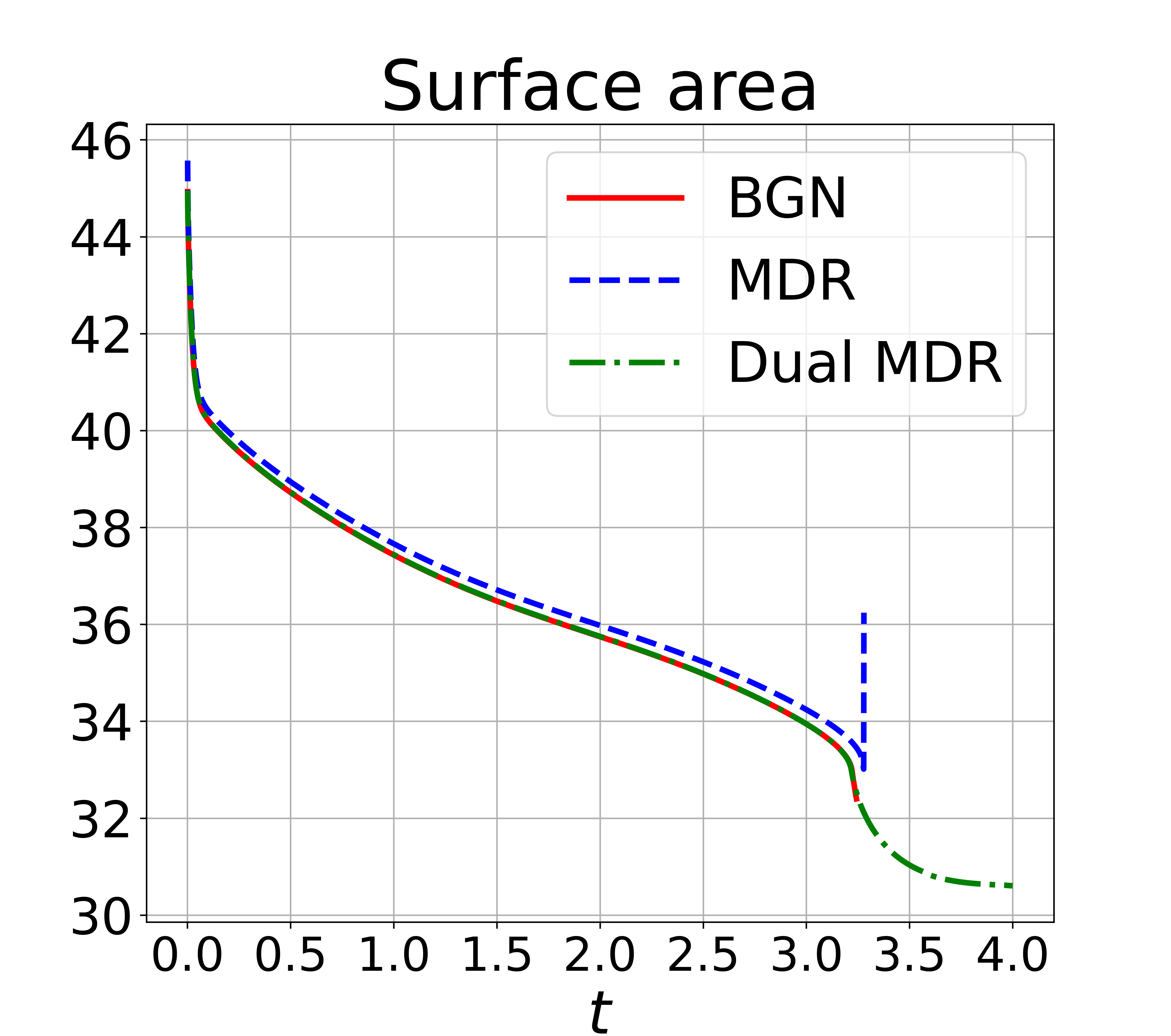}
  \caption{Surface area: $|\Gamma_h^m|$}
  \label{Example9d_90_area}
\end{subfigure}\hfill
\begin{subfigure}[t]{0.333\textwidth}
  \centering
  \includegraphics[width=\linewidth]{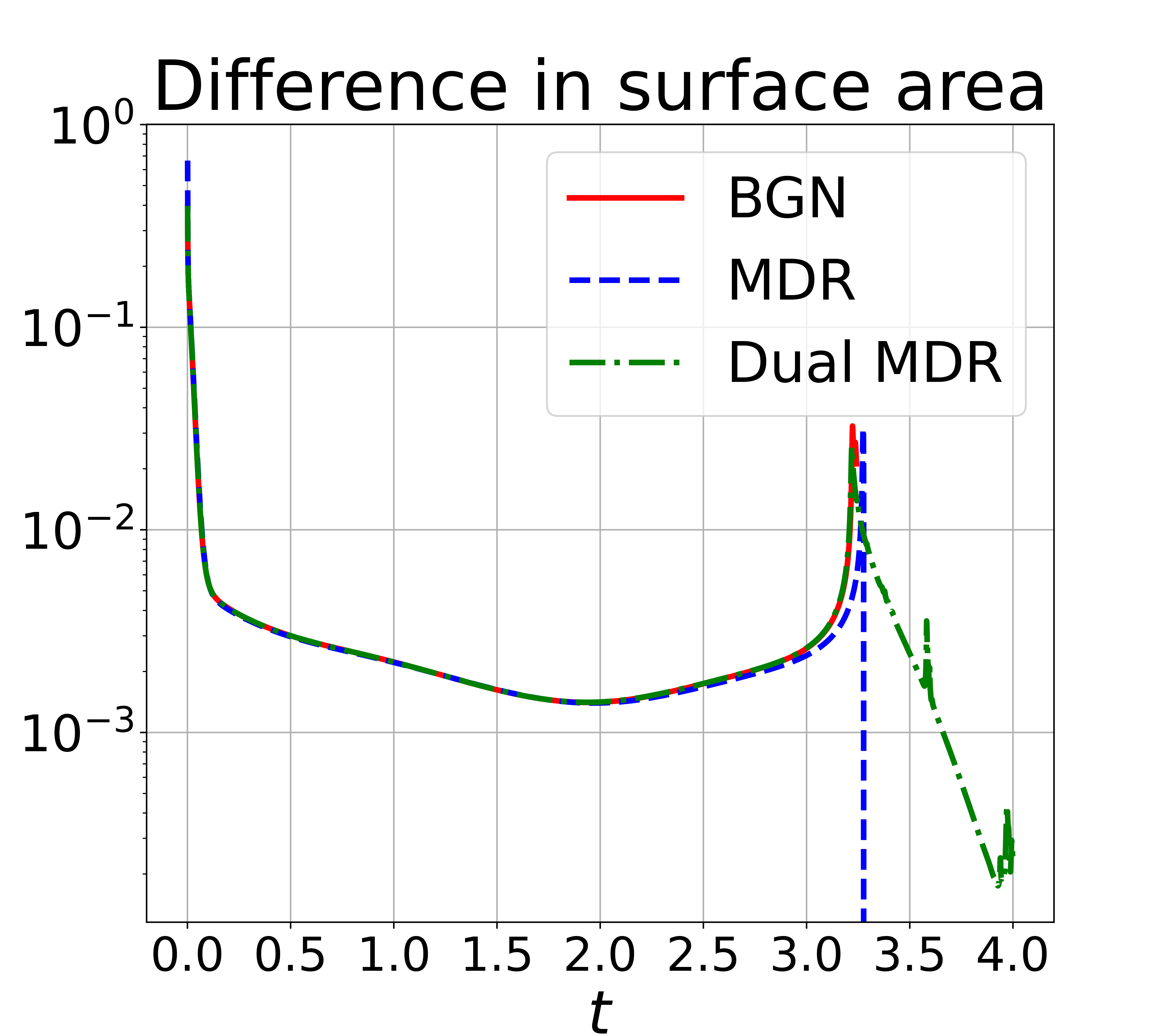}
  \caption{$|\Gamma_h^m| - |\Gamma_h^{m+1}|$}
  \label{Example9e_90_area_difference}
\end{subfigure}\hfill
\begin{subfigure}[t]{0.333\textwidth}
  \centering
  \includegraphics[width=\linewidth]{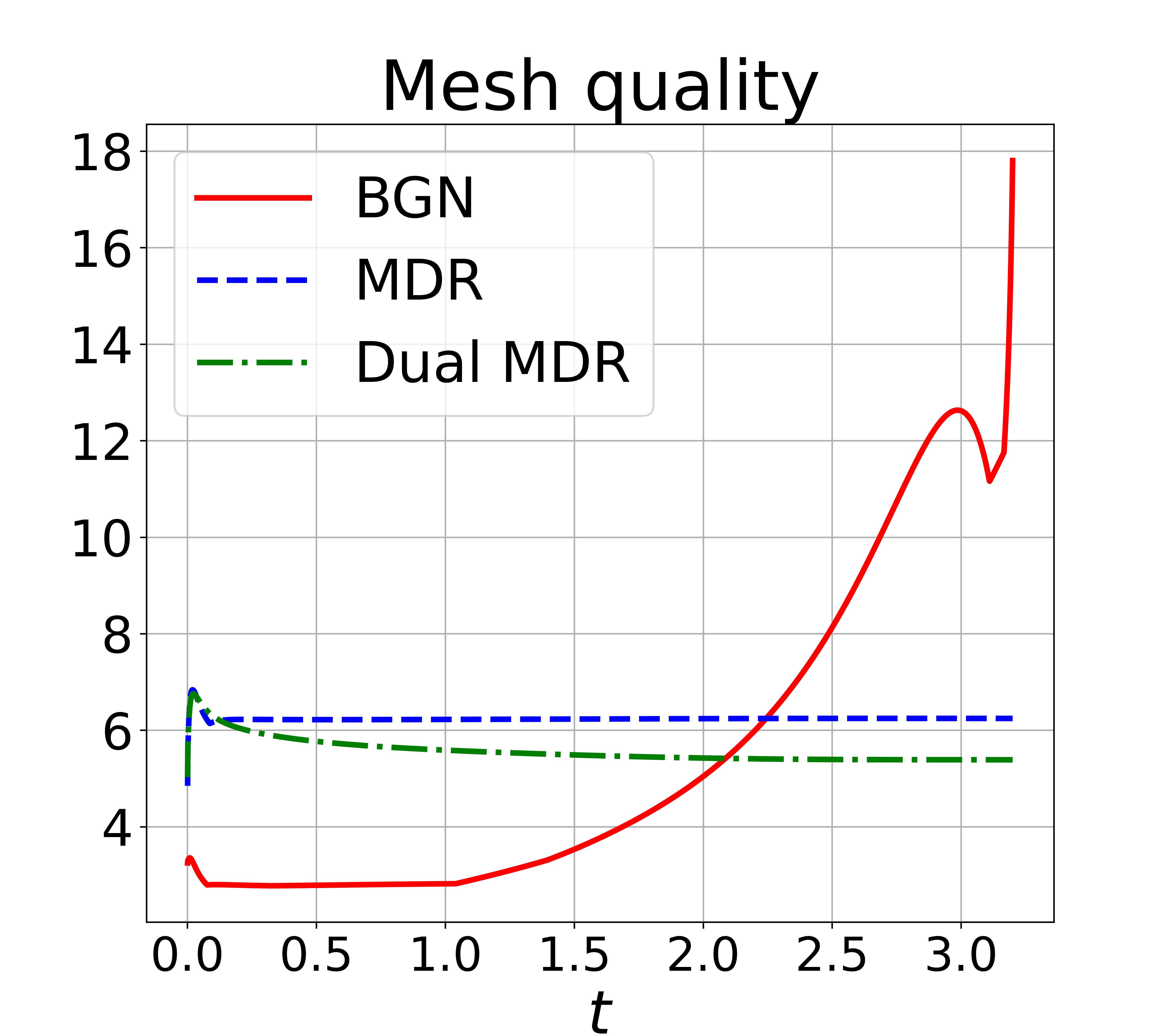}
  \caption{Mesh quality}
  \label{mesh_quality_example9}
\end{subfigure}

\caption{Comparison of the surface-area decay and mesh quality in Example~\ref{Example1-1-16}.}
\label{eg:Example9_90}
\end{figure}

\end{example}

\section{Conclusion}

By introducing a dual multiplier at the continuous level, we have constructed dual formulations for several curvature-driven geometric evolutions, including mean curvature flow, surface diffusion, and the solid-state dewetting problem. These dual formulations are equivalent to the original curvature flows, but they make the underlying energy structure explicit in a way that extends naturally to fully discrete, linearly implicit numerical schemes. As a principal application, we use the dual formulations to design linearly implicit and energy-stable methods for curvature flows equipped with the minimal-deformation-rate (MDR) tangential motion. The resulting dual-MDR schemes simultaneously maintain good mesh quality and guarantee energy stability. In addition, we have extended the approach to other tangential motions (see Section \ref{section:tangential}). Numerical experiments demonstrate that the proposed methods simultaneously maintain good mesh quality and guarantee energy stability over a wide range of time-step sizes, effectively addressing incompatible initial conditions in solid-state dewetting and accurately resolving pinch-off singularities, while preserving mesh quality up to the onset of the
singularity. Overall, the dual formulations introduced in this paper offer a novel framework for the design of structure-preserving numerical schemes for geometric evolution problems.

\renewcommand{\refname}{\bf References}
\bibliographystyle{abbrv}
\bibliography{main}

\end{document}